\documentclass[11pt, leqno]{amsart}
\usepackage{amsfonts}
\usepackage{amsfonts,latexsym,rawfonts,amsmath,amssymb,amsthm, a4,a4wide}
\usepackage{mathrsfs}
\usepackage[plainpages=false]{hyperref}

\usepackage{graphicx}

\RequirePackage{color}

\numberwithin{equation}{section}

\newcommand{\beq}{\begin{equation}}
\newcommand{\eeq}{\end{equation}}
\newcommand{\beqs}{\begin{eqnarray*}}
\newcommand{\eeqs}{\end{eqnarray*}}
\newcommand{\beqn}{\begin{eqnarray}}
\newcommand{\eeqn}{\end{eqnarray}}
\newcommand{\beqa}{\begin{array}}
\newcommand{\eeqa}{\end{array}}
\newcommand{\bb}{\text{\bf{b}}}

\newcommand{\R}{\mathbb R}

\newcommand{\abs}[1]{\left\vert#1\right\vert}

\newcommand{\e}{\varepsilon}
\newcommand{\p}{\partial}

\newcommand{\Om}{\Omega}

\newcommand{\bom}{{\overline\Om}}

{\begin{list}{}%
         {\setlength{\leftmargin}{#1}}%
         \item[]%
}
{\end{list}}

\newtheorem{prop}{Proposition}[section]
\newtheorem{thm}[prop]{Theorem}
\newtheorem{lem}[prop]{Lemma}

\newtheorem{rem}[prop]{Remark}

\newtheorem{defi}[prop]{Definition}

\renewcommand{\div}{\mbox{div}\,}
\newcommand{\trace}{\mbox{trace}\,}
\newcommand{\dist}{\text{dist}}

\newcommand{\norm}[1]{\left\Vert#1\right\Vert}

\author{Young Ho Kim}
\address{Department of Mathematics, Indiana University, 
Bloomington, IN 47405, USA}
\email {yk89@iu.edu}
\author{Nam Q. Le}
\address{Department of Mathematics, Indiana University, 
Bloomington, IN 47405, USA}
\email {nqle@iu.edu}
\author{Ling Wang}
\address{School of Mathematical Sciences, Peking University, Beijing, 100871, China}
\email{lingwang@stu.pku.edu.cn}
\author{Bin Zhou}
\address{School of Mathematical Sciences, Peking University, Beijing, 100871, China}
\email{bzhou@pku.edu.cn}
\thanks{Y.H. Kim and N. Q. Le were supported in part by NSF grant DMS-2054686.  L. Wang and B. Zhou were supported by National Key R$\&$D Program of China SQ2020YFA0712800 and NSFC grants 11822101}
\allowdisplaybreaks
\title[On the singular Abreu equations]{Singular Abreu equations and linearized Monge-Amp\`ere equations with drifts}

\makeatletter
\@namedef{subjclassname@2020}{%
  \textup{2020} Mathematics Subject Classification}
\makeatother

\begin{document}
\subjclass[2020]{35J40, 35J96, 35B65, 35B45}
\keywords{Singular Abreu equation, linearized Monge-Amp\`ere equation with drift, second boundary value problem, Monge-Amp\`ere equation, Legendre transform, pointwise H\"older estimate}
\begin{abstract}
We study the solvability of singular Abreu equations  which arise in the approximation of convex functionals subject to a convexity constraint. Previous works established the solvability of their second boundary value problems either in two dimensions, or in higher dimensions under either a smallness condition or a radial symmetry condition.
Here, we solve the higher dimensional case by 
transforming singular Abreu equations into linearized Monge-Amp\`ere equations with drifts. We establish
global H\"older estimates for the  linearized Monge-Amp\`ere equation with drifts under suitable hypotheses,
and then use them to the regularity and solvability of the second boundary value problem for singular Abreu equations in higher dimensions. 
Many cases with general right-hand side will also be discussed.
\end{abstract}

\maketitle

\section{Introduction and statements of the main results}
In this paper, we study the solvability of  the second boundary value problem of the following fourth order Monge-Amp\`ere type equation on a bounded, smooth, uniformly convex domain $\Omega\subset\R^n$ ($n\geq 2$):
\begin{equation}
\label{Abreu-sin}
  \left\{ 
  \begin{alignedat}{2}\sum_{i, j=1}^{n}U^{ij}D_{ij}w~& =-\gamma\div (|Du|^{q-2} Du)+ \bb\cdot Du+c(x, u):=f(x, u, Du, D^2u)~&&\text{\ in} ~\ \ \Omega, \\\
 w~&= (\det D^2 u)^{-1}~&&\text{\ in}~\ \ \Omega,\\\
u ~&=\varphi~&&\text{\ on}~\ \ \p \Omega,\\\
w ~&= \psi~&&\text{\ on}~\ \ \p \Omega.
\end{alignedat}
\right.
\end{equation}
Here  $\gamma\geq 0$, $q>1$, $U=(U^{ij})_{1\leq i, j\leq n}$ is the cofactor matrix of the Hessian matrix $$D^2 u=(D_{ij}u)_{1\leq i, j\leq n}\equiv \left(\frac{\p^2 u}{\p x_i \p x_j}\right)_{1\leq i, j\leq n}$$
of an unknown uniformly convex function $u\in C^2(\overline{\Omega})$;  $\varphi\in C^{3,1}(\overline\Omega)$, $\psi\in C^{1,1}(\overline\Omega)$, $\bb:\overline{\Omega}\rightarrow\R^n$ is a vector field on $\overline{\Omega}$, and $c(x, z)$ is a function on $\bom\times\mathbb R$. When the right-hand side $f$ depends only on the independent variable, that is $f=f(x)$, (\ref{Abreu-sin}) is the {\it Abreu equation} arising from the problem of finding extremal metrics on toric manifolds in K\"ahler geometry \cite{Ab}, and it is equivalent to
$$\sum_{i,j=1}^n\frac{\partial^2 u^{ij}}{\partial x_i\partial x_j}=f(x),$$
where $(u^{ij})$ is the inverse matrix of $D^2u$. 
The general form in \eqref{Abreu-sin} was introduced by the second author in \cite{LeCPAM, LeJAM, LePRS} in the study of convex functionals with a convexity constraint related to the Rochet-Chon\'e model \cite{RC} for the monopolist's problem in economics, whose Lagrangian depends on the gradient variable; see also Carlier-Radice \cite{CR} for the case where the Lagrangian does not depend on the gradient variable. 

More specifically, in the calculus of variations with a convexity constraint, one considers minimizers of  convex functionals
\[ \int_{\Omega} F_0(x, u(x), Du(x)) \,dx\]
among certain classes of convex competitors, where $F_0(x,z,\mathbf{p})$ is a function on $\overline{\Omega}\times \mathbb R\times \mathbb R^n$. 
One example is the Rochet-Chon\'e model with $q$-power ($q>1$) cost  
 \[F_{q,\gamma}(x,z,\mathbf{p})=(|\mathbf{p}|^q/q-x\cdot \mathbf{p}+z)\gamma(x),\] 
 where $\gamma$ is nonnegative and Lipschitz function called the relative frequency of agents in the population.

Since it is in general difficult to handle the convexity constraint, especially in numerical computations \cite{BCMO, Mir}, instead of analyzing these functionals directly, one might consider analyzing their perturbed versions by adding  
the penalizations $-\e\int_\Omega \log \det D^2u \,dx$ which are convex functionals in the class of $C^2$, strictly convex functions. The heuristic idea is that the logarithm of the Hessian determinant should act as a good barrier for the convexity constraint. This was verified numerically in \cite{BCMO} at a discretized level. 
Note that, critical points, with respect to compactly supported variations, of the convex functional
\[ \int_{\Omega} F_0(x, u(x), Du(x)) \,dx -\e\int_\Omega \log \det D^2u \,dx,\]
satisfy the Abreu-type equation
\[\e U^{ij} D_{ij} [(\det D^2 u)^{-1}]= -\sum_{i=1}^n \frac{\p}{\p x_i} \big(\frac{\p F_0}{\p p_i} (x, u, Du)\big) +  \frac{\p F_0}{\p z}(x, u, Du).\]
Here we denote $\mathbf{p} =(p_1,\ldots, p_n)\in\R^n$. 
In particular, for the Rochet-Chon\'e model with $q$-power ($q>1$) cost and unit frequency $\gamma\equiv 1$, that is, $F_0=F_{q,1}$, the above right-hand side is 
\[-\div (|Du|^{q-2} Du)+ n,\] which belongs to the class of right-hand sides considered in  \eqref{Abreu-sin}. When $F_0(x, z, \mathbf{p})=F(\mathbf{p}) + \hat F(x, z)$ 
the above right-hand side becomes \[-\div (DF(Du)) + \frac{\p \hat F}{\p z}(x, u).\]

 When $\gamma>0$,  we call \eqref{Abreu-sin} a {\it singular Abreu equation} because its right-hand side depends on $D^2u$ which can be just a matrix-valued measure for a merely convex function $u$.

Our focus in this paper will be on the case $\gamma>0$. For simplicity, we will take $\gamma=1$.

The Abreu type equations can be included in a class of fourth order Monge-Amp\`ere type equations of the form
\beq\label{4-eq-g}
U^{ij}D_{ij}[g(\det D^2 u)]=f
\eeq 
where $g:(0,\infty)\rightarrow (0,\infty)$ is an invertible function. In particular, when $g(t)=t^{\theta}$, one can take
  $\theta=-1$ and  $\theta=-\frac{n+1}{n+2}$ to get the Abreu type equation and the {\it affine mean curvature} type equation \cite{Ch}, respectively.
  It is convenient to write (\ref{4-eq-g}) as a system of two equations for $u$ and $w=g(\det D^2 u)$. One is a Monge-Amp\`ere equation for the convex function $u$ in the form of 
 \beq\label{MA1}
 \det D^2 u=
g^{-1}(w)
\eeq
 and other is the following linearized Monge-Amp\`ere equation for $w$:
\beq\label{LMA1}
U^{ij} D_{ij }w=f.
\eeq
The second-order linear operator  $\sum_{i, j=1}^nU^{ij}D_{ij}$ is the linearized Monge-Amp\`ere operator associated with the convex function $u$ because its coefficient matrix comes from linearizing the Monge-Amp\`ere operator:
\[U= \frac{\p\det D^2 u }{\p (D^2 u)}.\]
When $u$ is sufficiently smooth, such as $u\in W^{4, s}_{\text{loc}}(\Omega)$ where $s>n$, the expression $\displaystyle\sum_{i, j=1}^{n}U^{ij}D_{ij}w$ can be written as $\displaystyle\sum_{i, j=1}^{n}D_i(U^{ij}D_{j}w)$, since the cofactor matrix $(U^{ij})$ is divergence-free, that is, \[\displaystyle\sum_{i=1}^n D_i U^{ij}=0\] for all $j$.
 The regularity  and solvability of equation \eqref{4-eq-g}, under suitable boundary conditions, are closely related to the regularity theory  of the linearized Monge-Amp\`ere equation, initiated in the fundamental work of Caffarelli-Guti\'errez \cite{CG}. In the past  two decades, there have been many progresses on the study of these equations and related geometric problems, 
 including \cite{D1, D2, D3, D4, TW1, TW2, TW3, Z1, Z2, CW, CHLS, Le1, Le2}, to name a few.

According to the decomposition \eqref{MA1} and \eqref{LMA1}, a very natural boundary value problem for the class of fourth order equation (\ref{4-eq-g}) is the second boundary value problem where one describes the values of $u$ and $w$ on the boundary $\p\Omega$ as in \eqref{Abreu-sin}.  
 \subsection{Previous results and difficulties}
 \label{pre_res}
A summary of solvability results for (\ref{Abreu-sin}), or more generally, the second boundary value problem for (\ref{4-eq-g}), for the case $f\equiv f(x)$ is as follows. For the second boundary value problem of the affine mean curvature equation, that is, (\ref{4-eq-g}) with $g(t) = t^{-\frac{n+1}{n+2}}$,  Trudinger-Wang  \cite{TW2, TW3} proved the existence of a unique $C^{4,\alpha}(\overline\Omega)$ solution when $f\in C^\alpha(\overline\Omega)$ with $f\leq 0$, and 
a unique $W^{4,p}(\Omega)$ solution when $f\in L^{\infty}(\Omega)$ with $f\leq 0$. The analogous result for the Abreu equation (\ref{Abreu-sin}) was then obtained by the fourth author \cite{Z2}.
For the $W^{4,p}(\Omega)$ solution, the second author  \cite{Le1}  solved (\ref{Abreu-sin}) for $f\in L^p(\Omega)$ with $p>n$ and $f\leq 0$. The sign on $f$ was removed by Chau-Weinkove  \cite{CW} under the assumption that $f\in L^{p}(\Omega)$ with $p>n$ and $f^+:=\max\{f, 0\}\in L^{q}(\Omega)$ with $q>n+2$ for the affine mean curvature equation. Finally, in \cite{Le2}, the second author showed that the $W^{4,p}(\Omega)$ solution exists under the weakest assumption 
$f\in L^{p}(\Omega)$ with $p>n$ for a broad class of equations like \eqref{4-eq-g}, including both the affine mean curvature equation and the Abreu equation.
We will concentrate on the singular Abreu equation \eqref{Abreu-sin}, and its solvability in $C^{4,\alpha}$ and $W^{4, s}$ $(s>n)$ in this paper. We obtain solvability by establishing a priori higher order derivative estimates and then using the degree theory. 
 Essentially, establishing a priori estimates requires establishing the Hessian determinant estimates for $u$, and H\"older estimates for $w$.  

For the  singular Abreu equation,  the dependence of  the right-hand side on $D^2u$ creates two new difficulties in applying the regularity theory of the linearized Monge-Amp\`ere equation. The first difficulty lies in obtaining the \textit{a priori} lower and upper bounds for $\det D^2 u$, which is a critical step in 
applying the regularity results of the linearized Monge-Amp\`ere equation. The appearance of $D^2u$ has very subtle effects on the Hessian determinant estimates. The second author \cite{LeCPAM} obtained the Hessian determinant estimates for the case  $f=-\div (|Du|^{q-2} Du)$ in two dimensions with $q\geq 2$ by using a special algebraic structure of the equation. In a recent work 
of the second and the fourth authors \cite{LZ},
the Hessian determinant estimates for the case $1<q<2$ were established by using partial Legendre transform.
The second difficulty, granted that the bounds $0<\lambda\leq\det D^2 u\leq\Lambda<\infty$ have been established, consists in obtaining H\"older estimates for $w$ in the linearized Monge-Amp\`ere equation (\ref{LMA1}), which has no lower order terms on the left-hand side.  This requires certain integrability condition for the right-hand side, as can be seen from the simple equation $\Delta w=f$. In previous works  \cite{CG, GN1, GN2}, classical regularity estimates for linearized Monge-Amp\`ere equation were obtained for $L^n$ right-hand side. This integrability breaks down even in the case $f=- \Delta u$ (where $q=2$, $\bb=0$ and $c=0$), which is a priori at most $L^{1+\e}$ for some small constant $\e(\lambda,\Lambda, n)>0$ (see \cite{DFS, F, Sc}). With the H\"older estimates for the linearized Monge-Amp\`ere equation with $L^{n/2+\e}$ right-hand side in \cite{LN2}, the second author  \cite{LeCPAM} established the solvability of \eqref{Abreu-sin} for the case $f=-\div (|Du|^{q-2} Du)$ in two dimensions with $q\geq 2$.
When $1<q<2$, $f=-\div (|Du|^{q-2} Du)$ becomes more singular in $D^2 u$ and hence it has lower integrability (if any). 
However, in two dimensions, the second and the fourth authors \cite{LZ} solved the second boundary value problem \eqref{Abreu-sin} for
$f=-\div (|Du|^{q-2} Du)+c(x,u)$ for any $q>1$ under suitable assumptions on $c$ and the boundary data. The proof was based on the interior and global H\"older estimates for linearized Monge-Amp\`ere equation with the right-hand being the divergence of a bounded vector field which were established in \cite{LeCMP, LeJMAA}. The solvability of the singular Abreu equations (\ref{Abreu-sin}) in higher dimensions, even the simplest case $f=-\triangle u$, has been widely open. Only some partial results  were obtained in \cite{LeJAM} under either a smallness condition (such as replacing $f=-\Delta u$ by $f=-\delta \Delta u$ for a suitably small constant $\delta>0$) or a radial symmetry condition. 

\subsection{Statements of the main results}
The purpose of this paper is to solve the higher dimensional case of (\ref{Abreu-sin}). 
We will first consider the case that the right-hand side has no drift term $\bb\cdot Du$. This case answers in the affirmative the question raised in \cite[Page 6]{LZ}. In fact, we can establish the solvability for singular Abreu equations that are slightly more general than \eqref{Abreu-sin} where $\div (|Du|^{q-2} Du)$ is now replaced by
$\div (DF(Du))$ for a suitable convex function $F$. Our first main theorem states as follows.

\begin{thm}[Solvability of the second boundary value problem for singular Abreu equations in higher dimensions]
\label{thm:SBV-p-lap}
 Let $\Omega\subset\R^n$ be an open, smooth, bounded and uniformly convex domain.  Let $r>n$. Let $F\in W^{2, r}_{\text{loc}}(\R^n)$ be a convex function.
Assume that $\varphi\in C^{5}(\overline{\Omega})$ and $\psi\in C^{3}(\overline{\Omega})$ with $\min_{\p \Omega}\psi>0$. 
Consider the following second boundary value problem for a uniformly convex function $u$:
\begin{equation}
\label{eq:Abreu-plap}
  \left\{ 
  \begin{alignedat}{2}\sum_{i, j=1}^{n}U^{ij}D_{ij}w~& =-\div(DF(Du))+c(x, u)~&&\text{\ in} ~\ \ \Omega, \\\
 w~&= (\det D^2 u)^{-1}~&&\text{\ in}~\ \ \Omega,\\\
u ~&=\varphi~&&\text{\ on}~\ \ \p \Omega,\\\
w ~&= \psi~&&\text{\ on}~\ \ \p \Omega.
\end{alignedat}
\right.
\end{equation}
Here $(U^{ij})= (\det D^2 u)(D^2 u)^{-1}$, and $c(x, z)\leq 0$.
\begin{enumerate}
\item[(i)]  Assume $c\in C^\alpha(\overline\Omega\times \mathbb R)$ where $\alpha\in (0, 1)$. Then,  there exists a  uniformly convex solution $u\in W^{4, r}(\Omega)$ to \eqref{eq:Abreu-plap} with 
$$\|u\|_{W^{4,r}(\Omega)}\leq C$$
for some $C>0$ depending on $\Omega$, $n$, $\alpha$, $F$, $r$, $c$, $\varphi$ and $\psi$. 

Moreover, if $F\in C^{2,\alpha_0}(\R^n)$ where $\alpha_0\in (0, 1)$, then there exists a  uniformly convex solution $u\in C^{4, \beta}(\overline{\Omega})$ to \eqref{eq:Abreu-plap} where $\beta=\min\{\alpha,\alpha_0\}$ with 
$$\|u\|_{C^{4,\beta}(\overline{\Omega})}\leq C$$
for some $C>0$ depending on $\Omega$, $n$, $\alpha$, $\alpha_0$, $F$, $c$, $\varphi$ and $\psi$.

\item[(ii)]  Assume $c(x,z)\equiv c(x)\in L^p(\Omega)$ with $p>n$ where $c(x)\leq 0$. Then,  for  $s= \min\{r,p\}$, 
there exists a  uniformly convex solution 
$u\in  W^{4, s}(\Omega)$ to \eqref{eq:Abreu-plap} with 
$$\|u\|_{W^{4,s}(\Omega)}\leq C$$
for some  $C>0$  depending on $\Omega$, $n$, $p$, $F$, $r, s$, $\|c\|_{L^p(\Omega)}$, $\varphi$ and $\psi$. 
\end{enumerate}
\end{thm}
We will prove Theorem \ref{thm:SBV-p-lap} in Section \ref{SBV-plap-pf}.

We also discuss the solvability and regularity estimates of (\ref{Abreu-sin}) in the case that the right-hand side has  more general lower order terms
 and no sign restriction on $c$.  We mainly focus on the most typical case that the right-hand side has a Laplace term:
\begin{equation}
\label{Abreu-lap}
  \left\{ 
  \begin{alignedat}{2}\sum_{i, j=1}^{n}U^{ij}D_{ij}w~& =-\triangle u+ \bb\cdot Du+c(x,u)~&&\text{\ in} ~\ \ \Omega, \\\
 w~&= (\det D^2 u)^{-1}~&&\text{\ in}~\ \ \Omega,\\\
u ~&=\varphi~&&\text{\ on}~\ \ \p \Omega,\\\
w ~&= \psi~&&\text{\ on}~\ \ \p \Omega.
\end{alignedat}
\right.
\end{equation}
Here, $(U^{ij})= (\det D^2 u)(D^2 u)^{-1}$. 
Our second main result is the following theorem. 
 \begin{thm}[Solvability of the second boundary value problem for singular Abreu equations with lower order terms in high dimensions]
\label{SBV-1}
Let $\Omega\subset\R^n$($n\geq 3$) be an open, smooth, bounded and uniformly convex domain.  
Assume that $\varphi\in C^{5}(\overline{\Omega})$ and $\psi\in C^{3}(\overline{\Omega})$ with $\min_{\p \Omega}\psi>0$.
Consider  the 
second boundary value problem \eqref{Abreu-lap} with $c(x,z)\equiv c(x)$.
\begin{enumerate}
\item[(i)] If $\bb\in C^\alpha(\overline\Omega;\R^n)$ and $c\in C^\alpha(\overline\Omega)$ where $\alpha\in (0, 1)$, then there exists a  uniformly convex solution $u\in C^{4, \alpha}(\overline{\Omega})$ to \eqref{Abreu-lap} with 
$$\|u\|_{C^{4,\alpha}(\overline{\Omega})}\leq C$$
for some $C>0$ depending on $\Omega$, $n$, $\alpha$, $\|\bb\|_{C^\alpha(\overline\Omega)}$, $\|c\|_{C^\alpha(\overline\Omega)}$, $\varphi$ and $\psi$. 
\item[(ii)] If $\bb\in L^\infty(\Omega;\R^n)$ and $c\in L^p(\Omega)$ with $p>2n$, then there exists a  uniformly convex solution $u\in W^{4, p}(\Omega)$ to \eqref{Abreu-lap} with 
$$\|u\|_{W^{4,p}(\Omega)}\leq C$$
for some $C>0$  depending on $\Omega$, $n$, $p$, $\|\bb\|_{L^\infty(\Omega)}$, $\|c\|_{L^p(\Omega)}$, $\varphi$ and $\psi$.
\end{enumerate}
\end{thm}
We will prove Theorem \ref{SBV-1} in Section 
\ref{SBV-1-pf}. Furthermore, in two dimensions, when $\|\bb\|_{L^\infty(\Omega)}$ is small, depending on $\Omega, \psi$ and $\psi$, the conclusions of Theorem \ref{SBV-1} still hold; see Remark \ref{SBV-1-rem}.

The lack of non-positivity of $c$ in (\ref{Abreu-lap}) can raise more difficulties in the $L^\infty$-estimate and the use of Legendre transform in the Hessian determinant estimates. Compared to the weakest assumption  $c\in L^{p}(\Omega)$ with $p>n$ in \cite{Le2},
we need $p>2n$ in Theorem \ref{SBV-1}$(ii)$. However, in two dimensions, this assumption can be weakened provided stronger conditions on $\bb$ are imposed, but  $\|\bb\|_{L^\infty(\Omega)}$ can be arbitrarily large. This is the content of our final main result.

\begin{thm} [Solvability of the second boundary value problem for singular Abreu equations with lower order terms in two dimensions]
\label{SBV-2}
Let $\Omega\subset\R^2$ be an open, smooth, bounded and uniformly convex domain.  
Assume that $\varphi\in C^{5}(\overline{\Omega})$ and $\psi\in C^{3}(\overline{\Omega})$ with $\min_{\p \Omega}\psi>0$.
Consider  the 
second boundary value problem \eqref{Abreu-lap}.
Assume that $\bb\in C^1(\overline{\Omega};\R^n)$ with $\div(\bb)\leq \frac{32}{diam(\Omega)^2}$, and $c(x,z)\equiv c(x)\in L^p(\Omega)$ with $p>2$. Then there exists a  uniformly convex solution $u\in W^{4, p}(\Omega)$ to \eqref{Abreu-lap} with 
$$\|u\|_{W^{4,p}(\Omega)}\leq C$$
for some $C>0$  depending on $\Omega$, $p$, $\bb$, $\|c\|_{L^p(\Omega)}$, $\varphi$ and $\psi$.
\end{thm}
The proof of Theorem \ref{SBV-2} will be given in Section
\ref{SBV-2-pf}.
\begin{rem} Some remarks are in order.
\begin{enumerate}
\item[(1)] Theorem \ref{thm:SBV-p-lap} applies to all convex functions  $F(x)=|x|^q/q$ $(q>1)$ on $\R^n$ for which (\ref{eq:Abreu-plap}) becomes (\ref{Abreu-sin}) when $\bb=0$. Note that, if $1<q<2$, then $|x|^q\in W^{2, r}_{\text{loc}}(\R^n)$ for all $n<r<n/(2-q)$, while if $q\geq 2$, we have $|x|^q\in W^{2, r}_{\text{loc}}(\R^n)$ for all $r>n$.
\item[(2)]
By the Sobolev embedding theorem, the solutions $u$ obtained in our main results at least belong to $C^{3,\beta}(\overline{\Omega})$ for some $\beta>0$. 
\item [(3)] The condition $\div(\bb)\leq \frac{32}{diam(\Omega)^2}$ in Theorem \ref{SBV-2}  is due to the method of its proof  in obtaining a priori $L^{\infty}$ estimates that uses a Poincar\'e type inequality on planar convex domains in Lemma \ref{lem:Poincare}.
\end{enumerate}
\end{rem}

\begin{rem} We briefly relate the hypotheses in our existence results to concrete examples in applications. 
\begin{enumerate}
\item[(1)]Theorem \ref{thm:SBV-p-lap} applies to the approximation problem of the variational problem
\begin{equation}
\label{app-rem}
 \mathrm{inf}\int_{\Omega} F_0(x, u(x), Du(x)) \,dx\end{equation}
among certain classes of convex competitors, say, with the same boundary value $\varphi$ on $\p\Omega$,  where $F_0(x, z, \mathbf{p})= F( \mathbf{p}) + \hat F(x, z)$ with $F$ being convex and
$c(x, z)\equiv \frac{\p \hat F}{\p z} (x, z)\leq 0$.  The case $F\equiv 0$ is applicable. One particular example is $F_0(x, z, \mathbf{p})=\hat F(x, z)=  \big(|x|^2/2-z\big)\det D^2 v(x)$ where $v$ is a given function, which arises in wrinkling patterns in floating elastic shells in elasticity \cite{T}.

\item[(2)]Consider now $F_0(x, z, \mathbf{p})=\hat F(x, z)$. Denote $c(x, z)\equiv \frac{\p \hat F}{\p z} (x, z)$.
We note that without the condition $c(x, z)\leq 0$, \eqref{app-rem} might not have a minimizer. (For example, if $\hat F(x, z)= z^3$ so $c(x, z)=3z^2\geq 0$, then the infimum value of  \eqref{app-rem} is $-\infty$ if $\varphi\not\equiv 0$.)
On the other hand, when 
the assumption $c(x, z)\leq 0$ holds, a solution to \eqref{app-rem} always exists: One solution is the maximal convex extension of $\varphi$ from $\p\Omega$ to $\Omega$. The existence results in Theorem  \ref{thm:SBV-p-lap} imply that when $\hat F(x, z)$ is perturbed by convex functions of $Du$ (such as F(Du) where F is convex) and $\det D^2 u$ (such as $-\log\det D^2 u$), critical points of the resulting functionals, under appropriate boundary conditions, always exist, and this heuristically means that 
the resulting functionals continue to have minimizers.

\item[(3)]Theorem \ref{SBV-1} applies to  \eqref{Abreu-lap} with right-hand side $-\Delta u + 1$. This expression arises from the 
  Rochet-Chon\'e model with quadratic cost $F_0(x,z,\mathbf{p})=|\mathbf{p}|^2/2-x\cdot \mathbf{p}+z$, due to
 \[-\sum_{i=1}^n \frac{\p}{\p x_i} \big(\frac{\p F_0}{\p p_i} (x, u, Du)\big) +  \frac{\p F_0}{\p z}(x, u, Du)=-\Delta u + 1.\]
 \end{enumerate}

\end{rem}

\begin{rem}
Given our existence results concerning  \eqref{Abreu-sin}, one might wonder if the solutions found are unique. In general, for the fourth-order equations, 
we can not obtain the uniqueness of solutions by using the comparison principle.  However, for equations of the type \eqref{Abreu-sin}, we can obtain uniqueness in some special cases by exploring their very particular structures, using integral methods, and taking into account the concavity of the operator $\log\det D^2 u$ and the convexity of $|x|^q/q$ $(q>1)$ or $F$ in general. For example, we can infer from the arguments in 
 \cite[Lemma 4.5]{LeCPAM} that the uniqueness of  \eqref{Abreu-sin} holds when ${\bf b}\equiv 0$ and $c(x, z)$ satisfies the following monotonicity condition:
 \[(c(x, z)- c(x, \tilde z)) (z-\tilde z)\geq 0\quad\text{for all } x\in\overline{\Omega}\quad\text{and}\quad z, \tilde z\in\R.\]
 In particular, this implies that the solutions in Theorem \ref{thm:SBV-p-lap} (ii) are unique, and the solutions in Theorems \ref{SBV-1} and \ref{SBV-2} are unique provided that ${\bf b}\equiv 0$. To the best of our knowledge, the uniqueness for 
  \eqref{Abreu-sin} when ${\bf b}\neq 0$ is an interesting open issue.
\end{rem}

\subsection{On the proofs of the main results} Let us now say a few words about the proofs of our main results using a priori estimates and degree theory.  We focus on the most crucial point that overcomes the obstacles encountered in previous works: obtaining the a priori H\"older estimate for $w=(\det D^2 u)^{-1}$ in higher dimensions, once the Hessian determinant bounds on $u$ have been obtained. 
In this case, global H\"older estimates for $Du$ follow. Here, we use a new equivalent form (see Lemma \ref{low_lem}) for the singular Abreu equation to deal with the difficulties mentioned in Section \ref{pre_res}. 
In particular, in Theorem \ref{thm:SBV-p-lap}, instead of establishing the H\"older estimate for $w$, we establish the H\"older estimate for $\eta= w e^{F(Du)}$.  The key observation is that $\eta$ solves a linearized Monge-Amp\`ere equation with a drift term in which the very singular term \[\div (DF(Du))=\text{trace}(D^2 F(Du)D^2 u)\] no longer appears. 
Thus, the proof of Theorem \ref{thm:SBV-p-lap} reduces the global higher order derivative estimates for (\ref{eq:Abreu-plap}) to the global H\"older estimates of linearized Monge-Amp\`ere equations with drift terms.  To the best of the authors' knowledge, these  global H\"older estimates with full generality are not available in the literature. In the case of Theorems  \ref{SBV-1} and \ref{SBV-2}, the drift terms are also H\"older continuous. However, they do not vanish on the boundary and this seems to be difficult to prove H\"older estimates for $\eta$ at the boundary, not to 
mention global H\"older estimates. We overcome this difficulty by observing that {\it each of our singular Abreu equation is in fact equivalent to a family of linearized Monge-Amp\`ere equations with drifts}. In particular, at each boundary point $x_0$, 
\[\eta^{x_0}(x)= w(x) e^{F(Du(x))-DF(Du(x_0))\cdot (Du(x)-Du(x_0))-F(Du(x_0))}\]
solves a linearized Monge-Amp\`ere equation with a H\"older continuous drift that vanishes at $x_0$. This gives pointwise H\"older estimates for $\eta^{x_0}$ (and hence for $\eta$) at $x_0$. Combining this with interior H\"older estimates for linearized Monge-Amp\`ere equations with bounded drifts, we obtain the  global H\"older estimates for $\eta$ and hence for $w$. Section \ref{LOT} will discuss all these in detail.

For reader's convenience, we recall the following notion of pointwise H\"older continuity.
\begin{defi}[Pointwise H\"older continuity] A continuous function $v\in C(\overline{\Omega})$
 is said to be pointwise $C^\alpha$ ($0<\alpha<1$) at a boundary point $x_0\in\p\Omega$, if there exist constants $\delta, M>0$ such that
 \[|v(x)- v(x_0)|\leq M|x-x_0|^\alpha\quad\text{for all } x\in \Omega\cap B_\delta(x_0).\]
 \end{defi}
Throughout, we use the convention that repeated indices are summed.
\vskip 15pt
The rest of the paper is organized as follows. In Section \ref{diff}, we establish a new equivalent form for the singular Abreu equations which transform them into linearized Monge-Amp\`ere equations with drift terms, and the dual equations under Legendre transform. The global H\"older
estimates for the linearized Monge-Amp\`ere equation with drift terms, under suitable hypotheses, will be addressed in Section \ref{LOT}. With these estimates, we can prove Theorem \ref{thm:SBV-p-lap} in Section \ref{SBV-plap-pf}. The proofs of Theorems \ref{SBV-1} and \ref{SBV-2} will be given in Sections \ref{SBV-1-pf}, and \ref{SBV-2-pf}, respectively. In the final Section \ref{rem_sect}, we discuss (\ref{Abreu-sin})  with more general lower order terms,
and present a proof of Theorem \ref{LMA-G} on global H\"older estimates for solutions to the  linearized Monge-Amp\`ere equation with a drift term that are pointwise H\"older continuous at the boundary.

\vskip 30pt

\section[Equivalent forms of singular Abreu equations]{Equivalent forms of the singular Abreu equations}
\label{diff}

In this section, we derive some equivalent forms for the following general singular Abreu equations:
\begin{equation}\label{Abreu-plaplace}
\left\{
 \begin{alignedat}{2}
U^{ij} D_{ij} w&=- \div (DF(Du)) + Q(x, u, Du), ~&&\text{\ in} ~\ \ \Omega, \\\
 w~&= (\det D^2 u)^{-1}~&&\text{\ in}~\ \ \Omega,
 \end{alignedat}
\right.
\end{equation}
where $U=(U^{ij})= (\det D^2 u) (D^2 u)^{-1}$, $F\in W^{2, n}_{\text{loc}}(\R^n)$, and $Q$ is a function on $\mathbb R^n\times \mathbb R\times \mathbb R^n$. 

\subsection{Singular Abreu equations and linearized Monge-Amp\`ere equations with drifts}
Our key observation is the following lemma.
\begin{lem}[Equivalence of singular Abreu equations and linearized Monge-Amp\`ere equations with drifts]
\label{low_lem} Assume that a locally uniformly convex function $u\in W^{4, s}_{\text{loc}}(\Omega)$ ($s>n$) solves (\ref{Abreu-plaplace}).
Then 
\[\eta = w e^{F(Du)}\]
satisfies 
\begin{equation}\label{eq:F(p)-eq-n}
U^{ij}D_{ij} \eta -(\det D^2u) DF(Du)\cdot D \eta =e^{F(Du)}Q(x, u, Du).
\end{equation} 
\end{lem}
\begin{proof} Let $(u^{ij})= (D^2 u)^{-1} = wU$.
By computations using
$D_j U^{ij}= 0$ and $w=(\det D^2u)^{-1}$, we have
\[
U^{ij}D_{ij} w=D_j(U^{ij}D_i w)=D_j(u^{ij}D_i(\log w))=-D_j(u^{ij}D_i(\log\det D^2u))
\]
and
\beq\label{q-la}
D_j\left[u^{ij} D_i \left(F(Du)\right)\right]=\div (DF(Du)).
\eeq
It follows that equation \eqref{Abreu-plaplace} can be written as
\begin{equation}
\label{zetaQ}
D_j(u^{ij}D_i\zeta)=-Q(x, u, Du)
\end{equation}
where
 \begin{equation*}
 \zeta=\log\det D^2u-F(Du).
 \end{equation*}
 In other words,  in \eqref{Abreu-plaplace}, the singular term \[\div (DF(Du)) =\text{trace} (D^2 F(Du)D^2 u)\] can be absorbed into the left-hand side to turn it into a divergence form equation.

 Next, observe that $\zeta= -\log\eta$, and
 \[D_i \zeta = -D_i \eta/\eta = -D_i \eta \det D^2u e^{-F(Du)}.\]
Thus (\ref{zetaQ}) becomes
\begin{eqnarray*}
Q(x, u, Du) &=& - D_j(u^{ij}D_i \zeta) \\&=& D_j\Big(u^{ij} \det D^2u e^{-F(Du)} D_i\eta\Big)  \\&=& D_j\Big(U^{ij}  e^{-F(Du)} D_i \eta\Big) \\&=& U^{ij} D_j \Big( e^{-F(Du)} D_i\eta\Big) \,\,(\text{using the divergence free property of } (U^{ij})) \\
&=&U^{ij}D_{ij} \eta e^{-F(Du)} -U^{ij} D_i\eta e^{-F(Du)} D_kF (Du)D_{kj}u\\
&=& U^{ij} D_{ij}\eta e^{-F(Du)} -(\det D^2u) e^{-F(Du)}DF (Du)\cdot D \eta.
\end{eqnarray*}
Therefore, (\ref{eq:F(p)-eq-n}) holds, and the lemma is proved.
\end{proof}
\begin{rem} In general,  \eqref{Abreu-plaplace} is not the Euler-Lagrange equation of any functional. However, the introduction of \[\eta =w e^{F(Du)} = (\det D^2 u)^{-1} e^{F(Du)}\] in Lemma \ref{low_lem} has its root in an energy functional. Indeed, when $Q\equiv 0$,  \eqref{Abreu-plaplace} becomes
\[D_{ij} (U^{ij} (\det D^2 u)^{-1})+ \div (DF(Du))=0, \]
and this is the the Euler-Lagrange equation of the Monge-Amp\`ere type functional
\[\int_{\Omega} \Big(F(Du) -\log \det D^2 u\Big) dx = \int_\Omega \log \Big((\det D^2 u)^{-1} e^{F(Du)}\Big)\,dx.\]
\end{rem}
\begin{rem}
Taking $F(x)=|x|^{q}/q$ with $q>1$ in Lemma \ref{low_lem} where $x\in\R^n$, we find that an equivalent form of \[U^{ij}D_{ij} w=-\div (|Du|^{q-2} Du)+Q(x, u, Du), \quad w=(\det D^2u)^{-1}\] is
\begin{equation}\label{p-eq-n}
U^{ij}D_{ij}\eta-(\det D^2u)|Du|^{q-2}  Du \cdot D\eta=Q(x, u, Du)e^{\frac{|Du|^q}{q}},
\end{equation} 
where \[\eta = w e^{\frac{|Du|^q}{q}}.\]
\end{rem}
Lemma \ref{low_lem} shows that $\eta=w e^{F(Du)}$, where $u$ is a solution of (\ref{Abreu-plaplace}),  satisfies a linearized Monge-Amp\`ere equation with a drift term. This fact plays a crucial role in the study of
singular Abreu equations in higher dimensions in latter sections.
Once we have the determinant estimates for $\det D^2u$ for the second boundary value problem of (\ref{Abreu-plaplace}), we can estimate $u$ in  $C^{1,\alpha}(\overline{\Omega})$ provided the boundary data is smooth. This gives nice regularity properties for the right-hand side of (\ref{eq:F(p)-eq-n}) (and particularly, \eqref{p-eq-n}) and the drift 
on the left-hand side. Then the higher regularity estimates for \eqref{Abreu-plaplace} can be reduced to global H\"older estimates for the following linearized Monge-Amp\`ere equation with a drift term:
\begin{equation}\label{LMA-lower}
U^{ij} D_{ij}\eta+\bb\cdot D\eta+f(x)=0.
\end{equation} 
This is the content of Section \ref{LOT}.

\subsection{Singular Abreu equations under the Legendre transform}
In this section, we derive the dual equation of \eqref{Abreu-plaplace} under the Legendre transform in any dimension.
After the Legendre transform,  the dual equation is still a linearized Monge-Amp\`ere equation. 

Denote the Legendre transform $u^{\ast}$ of $u$ by $$u^{\ast}(y)=x \cdot Du-u, \quad \text{where }y=Du(x)\in \Omega^\ast =Du(\Omega). $$
Then
\[x= Du^* (y), \quad\text{and } u(x) = y\cdot Du^* (y) -u^*(y).\]
\begin{prop}[Dual equations for singular Abreu equations]
\label{new2}
Let $u\in W^{4, s}_{\text{loc}}(\Omega)$ ($s>n$)  be a uniformly convex solution to (\ref{Abreu-plaplace}) in $\Omega$. Then in $\Omega^*=Du(\Omega)$, its Legendre transform $u^{\ast}$ satisfies the following dual equation
\begin{eqnarray}\label{new-eq-leg}
 u^{\ast ij}D_{ij} \left(w^*+F(y)\right)=Q\left(D u^{*},y\cdot D u^{*}-u^*,y\right).
\end{eqnarray}
Here $( u^{\ast ij})$ is the inverse matrix of $D^2u^{\ast}$, and $w^*= \log\det D^2 u^{\ast}$.
\end{prop}

\begin{proof}
When (\ref{Abreu-plaplace}) is a Euler-Lagrange equation of a Monge-Amp\`ere type functional, we can derive its dual
equation from the dual functional as in \cite[Proposition 2.1]{LZ}. Here for the general case, 
we prove it by direct calculations. Note that for the case that the right-hand side has no singular term, the dual equation 
has been obtained in \cite[Lemma 2.7]{Le2}. We include a complete proof here for reader's convenience.

For simplicity, let $d=\operatorname{det} D^{2} u$ and $d^{*}=\operatorname{det} D^{2} u^{*}$. Then $d(x)=d^{*-1}(y)$ where $y= Du(x)$. We will simply write $d= d^{*-1}$ with this understanding. 

We denote by $\left(u^{i j}\right)$ and $\left(u^{* i j}\right)$ the inverses of the Hessian matrices $D^{2} u=\left(D_{i j} u\right)=\left(\frac{\partial^{2} u}{\partial x_{i} \partial x_{j}}\right)$ and $D^{2} u^{*}=\left(D_{i j}u^{*}\right)=\left(\frac{\partial^{2} u^{*}}{\partial y_{i} \partial y_{j}}\right)$, respectively. 
Let $(U^{\ast ij})=(\det D^2 u^*) (u^{\ast ij})$ be the cofactor matrix of $D^2 u^\ast$.

Note that $w={d}^{-1}=d^{*}$. Thus
\[
D_{j}w=\frac{\partial w}{\partial x_{j}}=\frac{\partial d^{*}}{\partial y_{k}} \frac{\partial y_{k}}{\partial x_{j}}=\frac{\partial d^{*}}{\partial y_{k}}  D_{k j}u=\frac{\partial d^{*}}{\partial y_{k}} u^{* k j} .
\]
Clearly,
\[
d^{*-1} \frac{\partial d^{*}}{\partial y_{k}}=\frac{\partial}{\partial y_{k}}\left(\log d^*\right)=D_{y_k} w^{*},
\]
from which it follows that \[D_{x_j}w=D_{y_k}w^{*}\left(U^{*}\right)^{k j}.\] Similarly, \[D_{i j} w=(\frac{\partial}{\partial y_{l}} D_{j} w) u^{* l i}.\] Hence, using
\[
U^{i j}=\operatorname{det} D^{2} u\cdot u^{i j}=\left(d^{*}\right)^{-1} D_{y_i y_j} u^{*},
\]
and the fact that $U^{*}=\left(U^{* i j}\right)$ is divergence-free, we obtain
\begin{eqnarray}
\label{Abu*}
U^{i j} D_{i j} w
&=&\left(d^{*}\right)^{-1} D_{y_i y_j} u^{*} u^{* l i} \frac{\partial}{\partial y_{l}} D_jw\nonumber\\
&=&\left(d^{*}\right)^{-1}(\frac{\partial}{\partial y_{j}} D_{j} w) \nonumber\\
&=&\left(d^{*}\right)^{-1} \frac{\partial}{\partial y_{j}}\left(D_{y_k} w^{*} U^{* k j}\right)\nonumber\\
&=&\left(d^{*}\right)^{-1} U^{* k j} D_{y_k y_j}w^{*}\nonumber\\
&=&u^{*ij}D_{ij}w^*.
\end{eqnarray}
On the other hand, by \eqref{q-la}, we have
\begin{eqnarray}
\label{Qu*}
\div (DF(Du))&=&D_{x_j}\left[u^{ij}D_{x_i}\left(F(Du)\right)\right]\nonumber\\
&=&u^{* l j} \frac{\partial}{\partial y_{l}}\left[u^*_{ij}u^{* ki} \frac{\partial}{\partial y_{k}}\left(F(y)\right)\right]\nonumber\\
&=&u^{\ast ij}D_{y_iy_j} \left(F(y)\right).
\end{eqnarray}
Combining (\ref{Abu*}) with (\ref{Qu*}) and recalling \eqref{Abreu-plaplace}, we obtain 
\[ u^{\ast ij}D_{ij} \left(w^*+F(y)\right) = Q(x, u(x), Du(x)) = Q\left(D u^{*},y\cdot D u^{*}-u^*,y\right), \]
which is
(\ref{new-eq-leg}).
The lemma is proved.
\end{proof}

\vskip 30pt

\section[Linearized Monge-Amp\`ere equations with drifts]{H\"older estimates for linearized Monge-Amp\`ere equation with drifts}
\label{LOT}

In this section, we study global H\"older estimates for  the linearized Monge-Amp\`ere equation with drift 
\begin{equation}
 \label{LMA-eq}
 \left\{
 \begin{alignedat}{2}
   U^{ij}D_{ij} v +\bb\cdot Dv~& = f ~&&\text{ in } ~ \Omega, \\\
v&= \varphi~&&\text{ on }~\p \Omega,
 \end{alignedat} 
  \right.
\end{equation}
where $U=(U^{ij})= (\det D^2 u)(D^2 u)^{-1}$ and $\bb:\Omega\rightarrow\R^n$ is a vector field. 

When there is no drift term, that is $\bb\equiv 0$, global H\"older estimates for (\ref{LMA-eq}) were established under suitable assumptions on the  bounds  $0<\lambda\leq \det D^2 u\leq\Lambda$ on the Hessian determinant of $u$, and the data. In particular, the case $f\in L^n(\Omega)$ was treated in \cite[Theorem 1.4]{Le1} (see also \cite[Theorem 4.1]{LN1} for a more localized version) and the case $f\in L^{n/2+\e}(\Omega)$ was treated in \cite[Theorem 1.7]{LN2}.

We would like to extend the above global H\"older estimates to the case with bounded drift. In this case, the interior H\"older estimates for (\ref{LMA-eq}) were obtained as a consequence of the interior Harnack inequality proved in \cite[Theorem 1.1]{LeCCM}. 
Note that Maldonado \cite{M} also proved a Harnack's inequality for (\ref{LMA-eq}) with different and stronger conditions on $\bb$.

Therefore, to obtain global H\"older estimates for (\ref{LMA-eq}) with a bounded drift $\bb$, it remains to prove the H\"older estimates at the boundary. Without further assumptions on $\bb$, this seems to be difficult with current techniques. However, when $\bb$ is pointwise H\"older continuous, and vanishes at a boundary point $x_0$, we can obtain the pointwise H\"older continuity of $v$ at $x_0$. 
This can be deduced from the following result, which is a drift version of \cite[Proposition 2.1]{Le1}.
\begin{prop}[Pointwise H\"older estimate at the boundary for solutions to non-uniformly elliptic, linear equations with pointwise H\"older continuous drift]\label{ptwH}
Assume that $\Omega\subset\R^n$ is a bounded, uniformly convex domain. Let
$\varphi\in C^{\alpha}(\p\Omega)$ for some $\alpha\in (0, 1)$, and $g\in L^n(\Omega)$. Assume that
the matrix $(a^{ij})$ is  measurable, positive definite and satisfies $\det (a^{ij})\geq \lambda$ in $\Omega$. Let $\bb\in L^{\infty}(\Omega;\mathbb R^n)$.
Let $v\in C(\overline{\Omega})\cap W^{2, n}_{loc}(\Omega)$ be the 
solution to 
$$a^{ij} D_{ij}v + \bb\cdot Dv~=g ~\text{in} ~\Omega, \quad
 v= \varphi ~\text{on}~ \p\Omega.
 $$
 Suppose there are constants $\mu,\tau\in (0, 1)$, and $M>0$ such that  at some $x_0\in\p\Omega$, we have
 \begin{equation}
 \label{bbvanish}
 |\bb (x)|\leq M|x-x_0|^\mu\quad \text{for all } x\in\Omega\cap B_{\tau}(x_0).\end{equation}
Then, there exist $\delta, C$ depending only on $\lambda, n, \alpha,\mu, \tau, M$, $\|\bb\|_{L^{\infty}(\Omega)}$, and $\Omega$ such that
\begin{equation*}|v(x)-v(x_{0})|\leq C|x-x_{0}|^{\frac{\min\{\alpha,\mu\}}{\min\{\alpha,\mu\} +4}}\left(\|\varphi\|_{C^{\alpha}(\p\Omega)} + \|g\|_{L^{n}(\Omega)}\right)~\text{for all}~ x\in \Omega\cap B_{\delta}(x_{0}). \end{equation*}
\end{prop}
We will prove Proposition \ref{ptwH} in Section \ref{ptwH_pf}.

Once we have the pointwise H\"older estimates at the boundary, global H\"older estimates for (\ref{LMA-eq}) follow. This is the content of the following theorem.
 \begin{thm}[Global H\"older estimates for solutions to the  linearized Monge-Amp\`ere equation with a drift term that are pointwise H\"older continuous at the boundary]
\label{LMA-G}
Assume that $\Omega\subset\R^n$ is a uniformly convex domain with boundary $\p\Omega\in C^3$.
Let $u \in C(\overline 
\Omega) 
\cap 
C^2(\Omega)$ be a convex function satisfying 
\[\lambda \leq \det D^2 u  \leq \Lambda \quad \text{in}\quad \Omega\]
for some positive constants $\lambda$ and $\Lambda$. Moreover, assume that $u|_{\p\Omega}\in C^3$. Let $(U^{ij})= (\det D^2 u)(D^2 u)^{-1}$. 
 Let $\bb\in L^{\infty}(\Omega;\R^n)$ with $\|\bb\|_{L^{\infty}(\Omega)}\leq M$,  $f\in L^{n}(\Omega)$ and $\varphi\in C^{\alpha}(\partial\Omega)$ for some $\alpha\in (0,1)$.
Assume that $v \in C(\overline{\Omega}) \cap W^{2,n}_{loc}( \Omega)$ is a  solution to the following linearized Monge-Amp\`ere equation with a drift term
\begin{equation*}
 \left\{
 \begin{alignedat}{2}
   U^{ij}D_{ij} v +\bb\cdot Dv ~& = f ~&&\text{ in } ~  \Omega, \\\
v&= \varphi~&&\text{ on }~\p \Omega.
 \end{alignedat} 
  \right.
\end{equation*} 
Suppose that there exist $\gamma\in (0, \alpha]$, $\delta>0$ and $K>0$ such that 
\begin{equation}
\label{Bdr_assumed}
|v(x)-v(x_{0})|\leq K|x-x_{0}|^{\gamma} \quad \text{for all}~ x_0\in\p\Omega, \text{ and } x\in \Omega\cap B_{\delta}(x_{0}). \end{equation}

Then, there exist a constant
 $\beta\in (0, 1)$ depending on $n$, $\lambda$, $\Lambda$, $\gamma$ and $M$,  and  a constant $C >0 $ depending only on  $\Omega$, $u|_{\p\Omega}$, $\lambda$, $\Lambda$, $n$,  $\alpha$, $\gamma$, $\delta$, $K$ and $M$ such that 
$$|v(x)-v(y)|\leq C|x-y|^{\beta}\Big( \|\varphi\|_{C^\alpha(\partial\Omega)}  + \|f\|_{L^{n}(\Omega)} \Big),~\forall x, y\in \Omega. $$
\end{thm}
The proof of Theorem \ref{LMA-G} is similar to that of \cite[Theorem 1.4]{Le1} for the case without a drift. For completeness and for reader's covenience, we present its proof at the end of the paper in Section \ref{rem_sect}.
\begin{rem}
It would be interesting to prove the global H\"older estimates in Theorem \ref{LMA-G} without the assumption (\ref{Bdr_assumed}).
\end{rem}
In Section \ref{wH_sec}, we will apply Theorem \ref{LMA-G} to establish the global H\"older estimates for Hessian determinants of singular Abreu equations provided that the Hessian determinants are bounded between two positive constants; see Theorem \ref{wHolder_thm}.

\subsection{Pointwise H\"older estimates at the boundary}
\label{ptwH_pf}
In this section, we prove Proposition \ref{ptwH}.
\begin{proof}[Proof of Proposition \ref{ptwH}] The proof is similar to that of \cite[Proposition 2.1]{Le1}. Due to the appearance of the drift $\bb$ and the pointwise H\"older continuity condition (\ref{bbvanish}), we include the proof for reader's convenience. 

Let \[K= \|\bb\|_{L^{\infty}(\Omega)}, \quad\text{and }L=\text{diam}(\Omega).\] In this proof, we fix the exponent  \[\gamma=\min\{\alpha,\mu\}/2.\]
However, the proof works for any exponent $\gamma\in (0, 1)$ such that $\gamma<\min\{\alpha,\mu\}$, and in this case, we replace the exponent $\frac{\min\{\alpha,\mu\}}{\min\{\alpha,\mu\} +4}$ in the proposition by $\frac{\gamma}{\gamma + 2}$.

Clearly $\varphi\in C^{\gamma}(\p\Omega)$ with $\|\varphi\|_{C^{\gamma}(\p\Omega)} \leq C(\alpha,\mu, L) \|\varphi\|_{C^{\alpha}(\p\Omega)}$. 
By considering the equation satisfied by $(\|\varphi\|_{C^{\gamma}(\p\Omega)} + \|g\|_{L^{n}(\Omega)})^{-1}v$,
we can assume that
$$\|\varphi\|_{C^{\gamma}(\p\Omega)} + \|g\|_{L^{n}(\Omega)}=1,$$
and it suffices to prove that, for some $\delta=\delta(n, \lambda, \alpha,\tau, K,  M,\mu,\Omega)>0$, we have
$$|v(x)-v(x_{0})|\leq C(n, \lambda, \alpha,\tau, K, M,\mu,\Omega)|x-x_{0}|^{\frac{\gamma}{\gamma +2}}~\text{for all}~ x\in \Omega\cap B_{\delta}(x_{0}). $$
Moreover, without loss of generality,  we assume that 
$$\Omega\subset \R^{n}\cap \{x_{n}>0\},~x_0=0\in\p\Omega.$$ 
Since $\det (a^{ij})\geq \lambda$, by the Aleksandrov-Bakelman-Pucci (ABP) estimate for elliptic, linear equations with drifts (see \cite[inequality (9.14)]{GT}), we have
\begin{eqnarray}\label{h-from-above}
\|v\|_{L^{\infty}(\Omega)}&\leq& \|\varphi\|_{L^{\infty}(\p\Omega)}  \nonumber\\&&+ \text{diam}(\Omega)\Bigg\{\exp\Big[\frac{2^{n-2}}{n^n \omega_n} \int_\Omega \Big( 1+ \frac{|\bb|^n}{\det (a^{ij})}\Big)dx \Big]-1\Bigg\}^{1/n}\Big \|\frac{g}{(\det (a^{ij}))^{1/n}}\Big\|_{L^n(\Omega)}
\nonumber\\&\leq&  C_0
\end{eqnarray}
for a constant $C_0(n,\lambda, K, L)>1$. Here we used $\omega_n=|B_1(0)|$, and $\|\varphi\|_{C^{\gamma}(\p\Omega)} + \|g\|_{L^{n}(\Omega)}=1$.
Hence, for any $\varepsilon \in (0,\tau^{\gamma})$
\begin{equation}|v(x)-v(0)\pm \e|\leq 3C_{0}:= C_{1}.
\label{3C0C1}
\end{equation}
Consider now the functions
$$\psi_{\pm}(x) := v(x)- v(0)\pm \e\pm C_{1} \kappa(\delta_2) x_{n}$$
where
\[\kappa (\delta_2):= (\inf \{y_{n}: y\in \overline{\Omega}\cap\partial B_{\delta_{2}}(0)\})^{-1}\]
in the region \[A:= \Omega\cap B_{\delta_{2}}(0)\] where $\delta_{2}<1$ is small to be chosen later. 

The uniform convexity of $\Omega$ gives
\begin{equation}
\inf \{y_{n}: y\in \overline{\Omega}\cap\partial B_{\delta_{2}}(0)\} \geq C_{2}^{-1}\delta^2_{2}
\end{equation}
where $C_2$ depends on the uniform convexity of $\Omega$. Thus,
\[\kappa(\delta_2) \leq C_2 \delta^{-2}_2.\]

Note that, if $x\in\partial \Omega$ with $|x|\leq \delta_{1}(\e):= \e^{1/\gamma}(\leq\tau)$ then, we have from $\|\varphi\|_{C^{\gamma}(\p\Omega)}\leq 1$ that
\begin{equation}
\label{bdr_vineq0}|v(x)-v(0)| =|\varphi(x)-\varphi(0)| \leq |x|^{\gamma} \leq \e.
\end{equation}
It follows that, if we choose $\delta_{2}\leq \delta_{1}$, then from (\ref{3C0C1}) and (\ref{bdr_vineq0}), we have
$$\psi_{-}\leq 0, \psi_{+}\geq 0~\text{on}~\partial A.$$
From (\ref{bbvanish}), we have
\[|\bb| \leq M\delta_2^\mu \quad\text{in } A,\]
and therefore
\[a^{ij}D_{ij}\psi_{-} + \bb\cdot D\psi_{-}= g - C_1 \kappa(\delta_2)\bb\cdot e_n\geq -|g|-C_1 C_2 M \delta_2^{\mu-2}\quad\text{in } A, \]
where $e_n=(0, \cdots, 0, 1)\in\R^n$.

Similarly,
\[a^{ij}D_{ij}\psi_{+} + \bb\cdot D\psi_{+}= g + C_1 \kappa(\delta_2)\bb\cdot e_n\leq |g| + C_1 C_2 M \delta_2^{\mu-2}\quad\text{in } A. \]
Again, applying the ABP estimate for elliptic, linear equations with drifts, we obtain
$$\psi_{-}\leq  C(n,\lambda,  K, L)\text{diam} (A) \|g + C_1 C_2 M \delta_2^{\mu-2}\|_{L^{n}(A)}\leq C_3(n, \lambda, K, M,\Omega, \tau, \mu)\delta^\mu_{2}~\text{in}~ A.$$
In the above inequality, we used $\|g\|_{L^{n}(A)} \leq 1$ and
\[\|g + C_1 C_2 M \delta^{\mu-2}\|_{L^{n}(A)} \leq \|g\|_{L^{n}(A)} +C_1 C_2 M \delta_2^{\mu-2} |A|^{1/n} \leq C(n, \lambda, K,  M,\Omega, \tau, \mu)\delta_2^{\mu-1}. \]
Similarly, we have $$
\psi_{+}\geq - C(n,\lambda,  K, L)\text{diam} (A) \|g+  C_1 C_2 M \delta_2^{\mu-2}\|_{L^{n}(A)}\geq  -C_3(n, \lambda, K, M,\Omega, \tau, \mu)\delta^\mu_{2}~\text{in}~ A.$$
We now restrict $\e\leq C_3^{\frac{-\gamma}{\mu-\gamma}}$ so that
$$\delta_{1} = \e^{1/\gamma}\leq [\e/C_3]^{1/\mu}.$$
Then, for $\delta_{2}\leq \delta_{1}$, we have $C_3\delta_{2}^\mu\leq \e$, and  thus, 
$$|v(x)-v(0)|\leq 2\e + C_{1} \kappa(\delta_2) x_{n}\quad\text{in } A.$$
Therefore, choosing $\delta_{2}= \delta_{1}$, we find
\begin{equation*}|v(x)-v(0)|\leq 2\e + C_{1} \kappa(\delta_2) x_{n}\leq 2\e + \frac{2C_{1}C_{2}}{\delta_{2}^2}x_{n}~\text{in}~ A.\end{equation*}
Summarizing, we obtain the following inequality
\begin{equation}
\label{op-ineq0}
|v(x)-v(0)|\leq 2\e + \frac{2C_{1}C_{2}}{\delta_{2}^2}|x| \leq 2\e + 2C_{1}C_{2}\e^{-2/\gamma}|x|
\end{equation}
for all $x,\e$ satisfying the following conditions
\begin{equation}
\label{xe-ineq}
|x|\leq \delta_{1}(\e):= \e^{1/\gamma}, \quad \e\leq C_3^{\frac{-\gamma}{\mu-\gamma}}: = c_{1}.
\end{equation}
Let us now choose 
$\e = |x|^{\frac{\gamma}{\gamma + 2}}.$
Then the conditions in (\ref{xe-ineq}) are satisfied as long as
$$|x|\leq \min\{c_{1}^{\frac{\gamma +2}{\gamma}}, 1\}:=\delta.$$
With this choice of $\delta$, and recalling (\ref{op-ineq0}), we have 
$$|v(x)-v(0)| \leq (2+ 2C_1 C_2)|x|^{\frac{\gamma}{\gamma + 2}}\quad \text{for all }x\in \Omega\cap B_{\delta}(0).$$
The proposition is proved.
\end{proof}

\subsection{Singular Abreu equations with Hessian determinant bounds}
\label{wH_sec}
In this section, we apply Theorem \ref{LMA-G} to establish the global H\"older estimates for Hessian determinants of singular Abreu equations provided that the Hessian determinants are bounded between two positive constants. This is the content of the following theorem.
\begin{thm}[H\"older continuity of Hessian determinant of  singular Abreu equations under Hessian determinant bounds]
\label{wHolder_thm}
Assume that $\Omega\subset\R^n$ is a uniformly convex domain with boundary $\p\Omega\in C^3$. Let $F\in  W^{2, r}_{\text{loc}}(\R^n)$ for some $r>n$, and let $g\in L^s(\Omega)$ where $s>n$. 
Let $\varphi\in C^4(\overline{\Omega})$ and $\psi\in C^2(\overline{\Omega})$ with $\min_{\p\Omega} \psi>0$. Assume that $u\in W^{4, s}(\Omega)$ is a uniformly convex solution to the singular Abreu equation:
\begin{equation*}
\left\{
 \begin{alignedat}{2}
U^{ij} D_{ij} w&=- \div (DF(Du))+ g(x), ~&&\text{\ in} ~\ \ \Omega, \\\
 w~&= (\det D^2 u)^{-1}~&&\text{\ in}~\ \ \Omega,\\\
  u~&= \varphi~&&\text{\ on}~\ \ \p\Omega,\\\
   w~&= \psi~&&\text{\ on}~\ \ \p\Omega,
 \end{alignedat}
\right.
\end{equation*}
where $U=(U^{ij})= (\det D^2 u) (D^2 u)^{-1}$. Suppose that, for some positive constants $\lambda$ and $\Lambda$, we have
\[\lambda \leq \det D^2 u  \leq \Lambda \quad \text{in}\quad \Omega.\]
Then, there exist 
constants $\beta, C >0 $ depending only on  $\Omega$, $\varphi, \psi$, $\lambda$, $\Lambda$, $n$,  $r$, $F$, and $\|g\|_{L^{n}(\Omega)}$, such that 
\[\|w\|_{C^{\beta}(\overline{\Omega})}\leq C. \]
\end{thm}
\begin{proof} Since $F\in W^{2, r}_{\text{loc}}(\R^n)$, by the Sobolev embedding theorem, we have $F\in C^{1, \alpha}(\R^n)$ where $\alpha=1-n/r\in (0, 1)$. From the Hessian determinant bounds on $u$, and $u=\varphi$ on $\p\Omega$ where $\varphi\in C^{4}(\overline{\Omega})$, 
 by \cite[Proposition 2.6]{LS}, we have 
 \begin{equation}
 \label{uc1a0}
 \|u\|_{C^{1,\alpha_0}(\overline\Omega)}\leq C_1,\end{equation}
 where $\alpha_0\in (0, 1)$ depends on $\lambda,\Lambda$, and $n$. The constant $C_1$ depends on $\Omega, n,\lambda,\Lambda$ and $\varphi$.
 
By Lemma \ref{low_lem}, 
 the function
 \[\eta(x)= w(x) e^{ F(Du(x))}\] satisfies
 \begin{equation}
 \label{etaeq}
 U^{ij} D_{ij}\eta -(\det D^2 u) DF(Du(x)) \cdot D\eta= g(x)e^{ F(Du(x))}\equiv f(x). \end{equation}
From (\ref{uc1a0}), we deduce that $\eta|_{\p\Omega}\in C^{\alpha_0}$ with estimate
\begin{equation}
\label{etabdr} \|\eta\|_{C^{\alpha_0}(\p\Omega)} \leq C_\ast(\psi, C_1, F).
\end{equation}
 
 {\it Step 1: Pointwise H\"older continuity of  $\eta$ at the boundary.} Fix $x_0\in\p\Omega$. 
 Let us denote 
 \[\tilde F(y):=F(y)-F(Du(x_0))-D F(Du(x_0))\cdot (y-Du(x_0))\quad\text{for } y\in\R^n.\]
 Then, we have
 \begin{equation*}U^{ij} D_{ij} w(x)=- \div (D\tilde F(Du(x)))  + g(x)\quad\text{in }\Omega.\end{equation*}
 By Lemma \ref{low_lem}, the function
 \[\eta^{x_0}(x)= w(x) e^{\tilde F(Du(x))}\] satisfies
 \begin{equation}
 \label{etax0}
 U^{ij} D_{ij}\eta^{x_0} -(\det D^2 u)  (DF (Du(x))- DF (Du(x_0))) \cdot D\eta^{x_0}= g(x)e^{\tilde F(Du(x))}\equiv f^{x_0}(x). \end{equation}
 Clearly,
 \begin{equation}
 \label{fx0Ln}
 \|f^{x_0}\|_{L^{n}(\Omega)} \leq  C_2,
 \end{equation}
 where $C_2$ depends on $\|F\|_{C^1(B_{C_1}(0))}$ and $\|g\|_{L^{n}(\Omega)}$.
 
 The vector field 
 \[\bb(x)=(\det D^2 u)\cdot (DF(Du(x))- DF (Du(x_0)))  \]
 satisfies in $\Omega$ the estimate
 \begin{equation}
\label{bbetax0} 
 |\bb(x)| \leq \Lambda \|DF\|_{C^\alpha(B_{C_1}(0))} |Du(x)-Du(x_0)|^\alpha \leq \Lambda C_1\|DF\|_{C^\alpha(B_{C_1}(0))}  |x-x_0|^{\alpha_1},\end{equation}
 where
 \[\alpha_1=\alpha\alpha_0.\]
 We also have $\eta^{x_0}\mid_{\p\Omega}\in C^{\alpha_1}(\p\Omega)$ with 
 \begin{equation}
 \label{etax0bdr}
 \|\eta^{x_0}\|_{C^{\alpha_1}(\p\Omega)} \leq C_3(\alpha, \alpha_0, C_1, \psi, \|DF\|_{C^\alpha(B_{C_1}(0))}).\end{equation}
 Note that 
 \[\det (U^{ij})=(\det D^2 u)^{n-1}\geq \lambda^{n-1}.\]
Hence, from (\ref{etax0}), (\ref{bbetax0}) and (\ref{etax0bdr}), we can apply Proposition \ref{ptwH}  and find constants \[\gamma=\alpha_1/(\alpha_1 +4)\in (0, 1),\] and $\delta, C_4>0$ depending only on $n,\lambda, \Lambda, \alpha, F,\varphi, \psi $, and $\Omega$ such that, for  all $x\in \Omega\cap B_{\delta}(x_{0})$,
\begin{equation}
\label{etax0x0}
|\eta^{x_0}(x)-\eta^{x_0}(x_{0})|\leq C_4|x-x_{0}|^{\gamma}\left( \|\eta^{x_0}\|_{C^{\alpha_1}(\p\Omega)}  + \|f^{x_0}\|_{L^{n}(\Omega)}\right)\leq C_5|x-x_{0}|^{\gamma},\end{equation}
where $C_5= C_4(C_2 + C_3)$.

Due to
\[\eta(x) = \eta^{x_0}(x) e^{F(Du(x_0))+DF(Du(x_0))\cdot (Du(x)-Du(x_0))},\]
and (\ref{uc1a0}),
(\ref{etax0x0}) implies the pointwise $C^\gamma$ continuity of $\eta$ at $x_0$ with estimate
\begin{equation}
\label{etaptw}
|\eta(x)-\eta(x_0)|\leq C_6|x-x_0|^\gamma \quad \text{for all } x\in \Omega\cap B_{\delta} (x_0),
\end{equation}
where $C_6$ depends on $\Omega$, $\varphi$, $\psi$,  $\lambda$, $\Lambda$, $n$,  $\alpha$, $F$ and $\|g\|_{L^n(\Omega)}$.

{\it Step 2: Global H\"older continuity of $\eta$ and $w$.} From (\ref{etaptw}), we can apply Theorem \ref{LMA-G} to (\ref{etaeq}) to conclude the global H\"older continuity of $\eta$. Since $w= \eta e^{-F(Du)}$, $w$ is also globally H\"older continuous. In other words,
there exist a constant
 $\beta\in (0, 1)$ depending on $n,\lambda,\Lambda,\alpha$ and $F$,  and  a constant $C >0 $ depending only on  $\Omega$, $\varphi$, $\psi$,  $\lambda$, $\Lambda$, $n$,  $r$, $F$ and $\|g\|_{L^n(\Omega)}$ such that 
$$\|w\|_{C^{\beta}(\overline{\Omega})}\leq C. $$
The theorem is proved.
\end{proof}


\vskip 30pt

\section{Proof of Theorem \ref{thm:SBV-p-lap}}
\label{SBV-plap-pf}

In this section, we prove Theorem \ref{thm:SBV-p-lap} using a priori estimates and degree theory. With Theorem \ref{wHolder_thm} at hand, a key step is to establish a priori Hessian determinant estimates for uniformly convex solutions $u\in W^{4, s}(\Omega)$ ($s>n$)  of (\ref{eq:Abreu-plap}).

For the Hessian determinant estimates, we will use the maximum principle and the Legendre transform; see also \cite[Theorem 1.2]{LZ} with a slightly different proof for the case of $F(x)=|x|^q/q$ $(q>1)$ and $c(x, z)$ being smooth.

\begin{lem}[Hessian determinant estimates]\label{lem:det-est-plap}
Let $\Omega\subset\R^n$ be an open, smooth, bounded and uniformly convex domain.  
Assume that $\varphi\in C^{5}(\overline{\Omega})$ and $\psi\in C^{3}(\overline{\Omega})$ with $\min_{\p \Omega}\psi>0$. Let $r, s>n$. Let $F\in W^{2, r}_{\text{loc}}(\R^n)$ be a convex function, and $c(x,z)$ be a function on $\overline\Omega\times\mathbb R$.  Suppose $c(x,z)\leq 0$ with 
 $c\in C^\alpha(\overline\Omega\times \mathbb R)$ where $\alpha\in (0, 1)$ or 
 $c(x,z)\equiv c(x)\in L^s(\Omega)$.
 Assume that $u\in W^{4, s}(\Omega)$ is a uniformly convex solution to the second boundary value problem 
\begin{equation*}
  \left\{ 
  \begin{alignedat}{2}U^{ij}D_{ij}w~& =-\div (DF(Du))+c(x, u)~&&\text{\ in} ~\ \ \Omega, \\\
 w~&= (\det D^2 u)^{-1}~&&\text{\ in}~\ \ \Omega,\\\
u ~&=\varphi~&&\text{\ on}~\ \ \p \Omega,\\\
w ~&= \psi~&&\text{\ on}~\ \ \p \Omega,
\end{alignedat}
\right.
\end{equation*} 
where $(U^{ij})=(\det D^2 u)(D^2 u)^{-1}$.
Then 
\[
C^{-1}\leq \det D^2u\leq  ( \min_{\partial\Omega}\psi)^{-1}\quad\text{in }\Omega,
\]
where $C>0$ is a constant depending on $\Omega$, $n$, $\varphi$, $\psi$, $F$ and $c$. In the case of $c(x,z)\equiv c(x)\in L^s(\Omega)$, the dependence of $C$ on $c$ is via $\|c\|_{L^{n}(\Omega)}$.
\end{lem}
\begin{proof}
From the convexity of $F$ and $u$, we have 
\[
-\div (DF(Du))=-\text{trace} (D^2 F(Du) D^2 u)\leq 0.
\]
This combined with $c(x,u)\leq 0$ yields \[U^{ij}D_{ij}w=-\div (DF(Du))+c(x,u)\leq 0\quad \text{in }\Omega.\] 
Hence, by the maximum principle, 
$w$ attains its minimum value in $\overline{\Omega}$ on the boundary. Thus \[w\geq\min_{\partial\Omega}w= \min_{\partial\Omega}\psi>0\quad \text{in }\Omega.\] This together with $\det D^2u=w^{-1}$ gives the upper bound for the Hessian determinant:
\[
\det D^2u\leq C_1:= ( \min_{\partial\Omega}\psi)^{-1}\quad\text{in }\Omega.
\]
From the above upper bound, by using $u=\varphi$ on $\partial\Omega$ together with $\Omega$ being smooth and uniformly convex, 
we can construct suitable barrier functions to deduce that 
\begin{equation}\label{eq:Du-bound-plap}
   \sup_{\Omega}|u| +  \|Du\|_{L^{\infty}(\Omega)}\leq C_2,
\end{equation}
where $C_2$ depends on $n, \varphi, \psi$ and $\Omega$.

We now proceed to establish a positive lower bound for the Hessian determinant.\\
Let \[u^*(y)= x\cdot Du(x)-u(x)\] be the Legendre transform of $u(x)$ where \[y = Du(x) \in \Omega^*:= Du(\Omega).\]
Then, (\ref{eq:Du-bound-plap}) implies
\begin{equation}\label{eq:u^*-bound-plap}
   \text{diam}(\Omega^\ast) +  \|u^*\|_{L^{\infty}(\Omega^\ast)}\leq C_3(n, \varphi, \psi, \Omega).
\end{equation}
In view of Proposition \ref{new2},  $u^*$ satisfies
\begin{equation}
\label{dualu}
u^{\ast ij}D_{ij} \left(w^*+F(y)\right)=c(Du^*,y\cdot Du^*-u^*)\quad \text{in }\Omega^\ast,
\end{equation}
where \[(u^{\ast ij})= (D^2 u^{\ast})^{-1},\quad \text{and } w^*=\log\det D^2u^*.\] 
Note that, for $y= Du(x)\in \p\Omega^\ast$ where $x\in\p\Omega$, we have
\[w^\ast(y) =\log (\det D^2 u(x))^{-1} =\log \psi(x). \]
By the ABP maximum principle applied to (\ref{dualu}), and recalling (\ref{eq:u^*-bound-plap}),  we find
\begin{eqnarray*}
\sup_{\Omega^*}(w^*+F(y)) &\leq& \sup_{\partial\Omega^*}(w^*+F(y))+C(n,   \text{diam}(\Omega^\ast))\left\|\frac{c(Du^*,y\cdot Du^*-u^*)}{{(\det D^2u^*)}^{-1/n}}\right\|_{L^n(\Omega^*)}\\ 
&=& \sup_{\partial\Omega^*}(w^*+F(y))+C(n,   \text{diam}(\Omega^\ast))\Bigg(\int_{\Omega}|c(x,u)|^n dx\Bigg)^{1/n}\\
&\leq& C_4 
\end{eqnarray*}
where $C_4$ depends on $\Omega$, $n$, $\varphi$, $\psi$, $F$ and $c$. Clearly, in the case of $c(x,z)\equiv c(x)\in L^s(\Omega)$, the dependence of $C_4$ on $c$ is via $\|c\|_{L^{n}(\Omega)}$. In the above estimates, we used
\begin{eqnarray*}
\left\|\frac{c(Du^*,y\cdot Du^*-u^*)}{{(\det D^2u^*)}^{-1/n}}\right\|_{L^n(\Omega^*)} &=&\Bigg(\int_{\Omega^*}|c(Du^*,y\cdot Du^*-u^*)|^n \det D^2u^*\,dy\Bigg)^{1/n}\\
&=&\Bigg(\int_{\Omega}|c(x,u)|^n \det D^2u^* \det D^2u\,dx\Bigg)^{1/n}\\
&=& \Bigg(\int_{\Omega}|c(x,u)|^n dx\Bigg)^{1/n}.
\end{eqnarray*}
It follows that
\[\sup_{\Omega^*}w^*(y)=\sup_{\Omega^*} \log\det D^2 u^\ast  \leq C_5\]
which implies 
\[
\det D^2u\geq e^{-C_5}>0 \quad\text{in }\Omega,
\]
where $C_5$ depends on $\Omega$, $n$, $\varphi$, $\psi$, $F$ and $c$. This is the desired positive lower bound for the Hessian determinant, and
the proof of the lemma is completed.
\end{proof}

Now, we can give the proof of Theorem \ref{thm:SBV-p-lap}.
\begin{proof}[Proof of Theorem \ref{thm:SBV-p-lap}] We divide the proof, using a priori estimates and degree theory, into three steps. Steps 1 and 2 establish higher order derivative estimates for $u\in W^{4, s}(\Omega)$ ($s>n$)  solutions. Step 3 confirms the existence of $W^{4, s}(\Omega)$ or  $C^{4,\beta}(\overline{\Omega})$ solutions via degree theory. 

In the following, we fix $s>n$ with the additional requirement that
\begin{equation*}
\left\{
 \begin{alignedat}{2}
s&= r~&&\text{\ in case }~ (i), \\\
  s~&= \min\{r, p\}  ~&&\text{\ in case }~(ii).
 \end{alignedat}
\right.
\end{equation*}

{\it Step 1: Determinant estimates and second order derivative estimates for uniformly convex $u\in W^{4, s}(\Omega)$ ($s>n$)  solutions $u$ of (\ref{eq:Abreu-plap}).}
  By Lemma \ref{lem:det-est-plap},  we have 
\beq\label{eq:d-boud-plap}
0<\lambda\leq \det D^2u\leq \Lambda:= ( \min_{\partial\Omega}\psi)^{-1} \quad\text{in }\Omega,
\eeq
 where $\lambda$ depends on $\Omega$, $n$, $F$, $\varphi$, $\psi$, and on either $c$ in case $(i)$, or $\|c\|_{L^{n}(\Omega)}$ in case $(ii)$. 
 
 From (\ref{eq:d-boud-plap}) and $u=\varphi$ on $\p\Omega$ where $\varphi\in C^{5}(\overline{\Omega})$, 
 by \cite[Proposition 2.6]{LS}, we have 
 \begin{equation}\label{C1beta}\|u\|_{C^{1,\alpha_0}(\overline\Omega)}\leq C_1,\end{equation}
 where $\alpha_0\in (0, 1)$ depends on $\lambda,\Lambda$, and $n$. The constant $C_1$ depends on $\Omega, n,\lambda,\Lambda$ and $\varphi$.
 
With (\ref{eq:d-boud-plap}) and $F\in W^{2, r}_{\text{loc}}(\R^n)$, we can use Theorem \ref{wHolder_thm} to find  $\beta_0\in (0,1)$, and $C_5>0$ depending on $\Omega, n, F, r, \varphi, \psi, c$, such that
\[
\|w\|_{C^{\beta_0}(\overline{\Omega})}\leq C_2(\Omega, n, F, r, \varphi, \psi, c).
\]
Hence
$
\det D^2u=w^{-1}\in C^{\beta_0}(\overline{\Omega})$.
By the global Schauder estimates for the Monge-Amp\`ere equation in  \cite{S2,TW3}, we  have
\begin{equation}
\label{ue2beta}
\|u\|_{C^{2,\beta_0}(\overline{\Omega})}\leq C_3(\Omega, n, F, r, \varphi, \psi, c).
\end{equation}
Combining this with \eqref{eq:d-boud-plap}, we find \[C_4^{-1} I_n\leq D^2u\leq  C_4 I_n \quad\text{in }\Omega\] for some $C_4(\Omega, n, F, r, \varphi, \psi, c)>0$.  Here $I_n$ denotes the identity $n\times n$ matrix. In other words, the linear operator $U^{ij}D_{ij}$ is uniformly elliptic with  coefficients $U^{ij}$ bounded in $C^{\beta_0}(\overline{\Omega})$.
\vskip 7pt

{\it Step 2:   Global higher order derivative estimates for uniformly convex $W^{4, s}(\Omega)$ $(s>n)$ solutions $u$ of (\ref{eq:Abreu-plap}).} Denote the right-hand side of \eqref{eq:Abreu-plap} by 
\begin{equation}\label{eq:f_e}
f:=-\div (DF(Du)) + c(x, u)=-\text{trace} (D^2 F(Du) D^2 u) +c(x,u).
\end{equation}
Observe that, 
one has 
the following estimate
\begin{equation}
\label{GLr}
\|\text{trace} (D^2 F(Du) D^2 u)\|_{L^{r}(\Omega)} \leq C(\Omega, n, F, r, \varphi, \psi, c).\end{equation}
 Indeed,  we have
 \begin{eqnarray*}
 \|\text{trace} (D^2 F(Du) D^2 u)\|^r_{L^r(\Omega)} &\leq& n^2 \|D^2 u\|^r_{L^{\infty}(\Omega)}\| D^2 F(Du) \|^r_{L^r(\Omega)} \\
 &\leq& n^2C_3^r \int_{\Omega} |D^2 F(Du(x))|^{r} dx\quad (\text{using } (\ref{ue2beta}))\\
 &=& n^2C_3^r \int_{Du(\Omega)}  |D^2 F(y)|^{r}  \frac{1}{\det D^2 u((Du)^{-1}(y))}dy \\
 &\leq& n^2C_3^r \lambda^{-1} \int_{B_{C_1}(0)} |D^2 F(y)|^{r} dy \quad (\text{using } (\ref{eq:d-boud-plap}) \text{and } (\ref{C1beta}))\\
 &\leq& C_3^r \lambda^{-1} C(n, C_1, F, r).
 \end{eqnarray*}

 We consider cases $(i)$ and $(ii)$ separately. 
\vskip 7pt
\begin{enumerate}
\item[(i)]
{\it  The case of $ c\in C^\alpha(\overline\Omega\times\mathbb R)$.} Recall that $s=r$ in this case.
We have from (\ref{GLr}) that $f=-\text{trace} (D^2 F(Du) D^2 u) + c(x, u)\in L^s(\Omega)$ with estimate
\[\|f\|_{L^{s}(\Omega)} \leq C(\Omega, n, F, r, s, \varphi, \psi, c).\]
By {\it Step 1}, \begin{equation*}
  U^{ij}D_{ij}w =f~\text{\ in} ~ \Omega, 
w ~=\psi~\text{\ on}~\p \Omega,
\end{equation*} 
is a uniformly elliptic equation in $w$ with $C^{\beta_0}(\overline{\Omega})$ coefficients. Thus, from the standard $W^{2,p}$ theory for uniformly elliptic linear equations (see \cite[Chapter 9]{GT}),
we obtain the following $W^{2, s}(\Omega)$ estimate:
\[
\|w\|_{W^{2,s}(\Omega)}\leq C(\Omega, n, q, s, \varphi, \psi, c).
\]
Now, recalling  $\det D^2 u=w^{-1}$ in $\Omega$ with $u=\varphi$ on $\partial\Omega$, we can differentiate and apply the standard Schauder and Calderon-Zygmund theories to obtain the following global
$W^{4,s}$ estimate of $u$:
$$\|u\|_{W^{4,s}(\Omega)}\leq C(\Omega, n, F, r, s, \varphi, \psi, c).$$
Indeed, for any $k\in\{1, \ldots, n\}$
by differentiating $\det D^2 u=w^{-1}$ in the $x_k$ direction, we see that $D_k u$
solves the equation
\[U^{ij} D_{ij} (D_k u)= D_k (w^{-1})\in W^{1, s} (\Omega),\]
which is uniformly elliptic with $C^{\beta_0}(\overline\Omega)$ coefficients $U^{ij}$ due to \eqref{eq:d-boud-plap} and \eqref{ue2beta}. Since $s>n$, we have $W^{1, s} (\Omega)\in C^{0, 1-n/s}(\overline{\Omega})$. By the classical Schauder theory (see \cite[Chapter 6]{GT} for example), we deduce that $D_k u\in C^{2, \beta_1}(\overline{\Omega})$ for all $k$
with appropriate estimates, where $\beta_1 =\min\{\beta_0, 1-n/s\}$. This shows that $u\in C^{3, \beta_1}(\overline{\Omega})$ and the coefficients satisfy $U^{ij}\in C^{1, \beta_1}(\overline{\Omega})$. Next, for any  $l\in\{1, \ldots, n\}$, we differentiate the preceding equation in the $x_l$ direction to get
\[U^{ij} D_{ij} (D_{kl} u) = D_{kl} (w^{-1}) - D_l U^{ij} D_{ijk}u \in L^s(\Omega)\quad\text{for all } k, l\in \{1, \ldots, n\}. \]
Applying the Calderon-Zygmund estimates, we obtain $D_{kl} u \in W^{2, s}(\Omega)$ for all $k, l\in \{1, \ldots, n\}$ with appropriate estimates.  Consequently, $u\in W^{4, s}(\Omega)$ with estimate stated above.

Moreover, in the particular case that $F\in C^{2,\alpha_0}(\R^n)$, we find that $f\in C^{\gamma}(\overline{\Omega})$ where $\gamma\in(0, 1)$ depends only on $\alpha$, $F, \alpha_0$,  and $\beta_0$ with estimate
 \begin{equation}
 \label{fcgamma}
   \|f\|_{C^{\gamma}(\overline{\Omega})}\leq C(\Omega, n, \alpha, q, c, \varphi, \psi).\end{equation}
 Thus, we can apply the classical Schauder theory (see \cite[Chapter 6]{GT} for example) to \eqref{eq:Abreu-plap} which, by {\it Step 1}, is a uniformly elliptic equation in $w$ with $C^{\beta_0}(\overline{\Omega})$ coefficients. We conclude that $w \in C^{2, \beta}(\overline{\Omega})$,  where $\beta\in (0, 1)$ depends only on $n,\gamma$ and $\beta_0$, with estimate
  \[\|w\|_{C^{2, \beta}(\overline{\Omega})}\leq C(\Omega, n, \alpha, \alpha_0, F, c, \varphi, \psi).\]
  Due to \[\det D^2 u=w^{-1} \quad \text{in }\Omega, \quad u=\varphi \quad 
 \text{on }\partial\Omega,\] this implies that $u\in C^{4, \beta}(\overline{\Omega})$ with estimate
   \begin{equation}
   \label{uc4b}
 \|u\|_{C^{4, \beta}(\overline{\Omega})}\leq C(\Omega, n, \alpha, \alpha_0, F, c, \varphi, \psi).
 \end{equation}
With this estimate, we go back to   $f=-\text{trace} (D^2 F(Du) D^2 u) + c(x, u)$ and find that one can actually take $\gamma=\min\{\alpha,\alpha_0\}$ in (\ref{fcgamma}). Repeating the above process, one  find that (\ref{uc4b}) holds for $\beta=\min\{\alpha,\alpha_0\}$.

\item[(ii)]{\it  The case of $c(x,z)\equiv c(x) \in L^p(\Omega)$ with $p>n$.} Recall that in this case $s=\min\{r, p\}$. Then, we have from \eqref{eq:f_e} and (\ref{GLr}) that
 \begin{equation*}\|f\|_{L^{s}(\Omega)} \leq (\Omega, n, p, F, r, s, \varphi, \psi, \|c\|_{L^p(\Omega)}).\end{equation*}
 Arguing as in the case $(i)$ above, we obtain the following $W^{4,s}$ estimate of $u$:
$$\|u\|_{W^{4,s}(\Omega)}\leq C(\Omega, n, p, F, r, s,\varphi, \psi, \|c\|_{L^p(\Omega)}).$$
\end{enumerate}

 {\it Step 3: Existence of solutions via degree theory.} From the $C^{4,\beta}(\overline{\Omega})$ or $W^{4, s}(\Omega)$ estimates for uniformly convex $W^{4, s}(\Omega)$ solutions $u$ of (\ref{eq:Abreu-plap}) in {\it Step 2}, we can use the Leray-Schauder degree theory as in \cite{CW,TW2,LeCPAM}
to prove the existence of $C^{4,\beta}(\overline{\Omega})$ or $W^{4, s}(\Omega)$ solutions to (\ref{eq:Abreu-plap}) as stated in the theorem. We omit details here.
\end{proof}


\vskip 10pt

\section{Proof of Theorem \ref{SBV-1}}
\label{SBV-1-pf}

In this section, we prove Theorem \ref{SBV-1}.  
As in the  proof of Theorem \ref{thm:SBV-p-lap} in Section \ref{SBV-plap-pf}, we focus on a priori estimates for smooth, uniformly convex solutions. The most crucial ones are the Hessian determinant estimates.
Without the sign of $c$, we first  need to obtain the \textit{a priori} $L^\infty$-bound for $u$.

\begin{lem}[A priori $L^\infty$-bound for uniformly convex $W^{4, n}$ solutions]
\label{lem:bound-u-high}
Let $\Omega\subset\R^n$($n\geq 3$) be an open, smooth, bounded and uniformly convex domain.  
Assume that $\varphi\in C^{5}(\overline{\Omega})$ and $\psi\in C^{3}(\overline{\Omega})$ with $\min_{\p \Omega}\psi>0$. 
Assume $\bb\in L^\infty(\Omega;\R^n)$. Suppose that there exist functions $g_1, g_2\in L^1(\Omega)$ and a constant $0\leq m<n-1$ such that 
\beq\label{c-cond}
|c(x,z)|\leq |g_1(x)|+|g_2(x)|\cdot  |z|^{m} \quad\text{in }\Omega\times\R.
\eeq
Assume that $u\in W^{4, n}(\Omega)$ is a uniformly convex solution to \eqref{Abreu-lap}. Then there exists a constant $C>0$ depending on $\Omega$, $n$, $\varphi$, $\psi$, $\|\bb\|_{L^\infty(\Omega)}$, $\|g_1\|_{L^1(\Omega)}$, $\|g_2\|_{L^1(\Omega)}$ and $m$ such that
\[
\|u\|_{L^\infty(\Omega)}\leq C.
\]
\end{lem}

\begin{proof} From $u\in W^{4, n}(\Omega)$ and the Sobolev embedding theorem, we have $u\in C^2(\overline{\Omega})$.
For a convex function $u\in C^2(\Omega)$ with $u = \varphi$ on $\partial\Omega$, we have (see, e.g., \cite[inequality (2.7)]{Le2})
\begin{equation}\label{eq:u-bound-high}
\|u\|_{L^{\infty}(\Omega)} \leq \|\varphi\|_{L^{\infty}(\Omega)}+C_1\left(n, \Omega,\|\varphi\|_{C^{2}(\Omega)}\right)\left(\int_{\partial \Omega}\left(u_{\nu}^{+}\right)^{n}\,dS\right)^{1 / n},
\end{equation}
 where  $u_{\nu}^{+}=\max \left(0, u_{\nu}\right)$, $\nu$ is the unit outer normal of $\p\Omega$ and $dS$ is the boundary measure. 
Thus, to prove the lemma, it suffices to prove
\[
\int_{\partial \Omega}\left(u_{\nu}^{+}\right)^{n} d S \leq C(\Omega, n, \varphi, \psi, \|\bb\|_{L^\infty(\Omega)}, \|g_1\|_{L^1(\Omega)}, \|g_2\|_{L^1(\Omega)}, m).
\]
For this, we use the arguments as in the proof of \cite[Lemma 4.2]{LeCPAM}. 
Observe that,
since $u$ is convex with boundary value $\varphi$ on $\p\Omega$, we have
$u_\nu\geq -\|D\varphi\|_{L^{\infty}(\Omega)}$ and hence 
\begin{equation}
\label{unuab}
|u_\nu|\leq u_\nu^{+}+ \|D\varphi\|_{L^{\infty}(\Omega)},\quad \text{and }
(u_\nu^{+})^n\leq u^n_\nu + \|D\varphi\|^n_{L^{\infty}(\Omega)} \quad\text{on }\p\Omega.
\end{equation}

Let $\rho$ be a strictly convex defining function of $\Omega$, i.e.
\[
\Omega:=\left\{x \in \R^{n}: \rho(x)<0\right\}, \rho=0 \text { on } \partial \Omega \text { and } D \rho \neq 0 \text { on } \partial \Omega.
\] 
Let
\[
\tilde{u}=\varphi+\mu\left(e^{\rho}-1\right) .
\]
Then, for $\mu$ large, depending on $n$, $\Omega$ and $\|\varphi\|_{C^{2}(\overline{\Omega})}$, the function $\tilde{u}$ is uniformly convex, belongs to $C^{5}(\overline{\Omega})$. Furthermore, as in \cite[Lemma 2.1]{Le2}, there exists a constant $C_2>0$ depending only on $n$, $\Omega$, and $\|\varphi\|_{C^{4}(\overline{\Omega})}$ such that the following facts hold:

(i) $\|\tilde{u}\|_{C^{4}(\overline{\Omega})} \leq C_2, \quad$ and $\operatorname{det} D^{2} \tilde{u} \geq C_2^{-1}>0$ in $\Omega$,

(ii) letting $\tilde{w}=\left[\operatorname{det} D^{2} \tilde{u}\right]^{-1}$, and denoting by $\left(\tilde{U}^{i j}\right)$ the cofactor matrix of $D^{2} \tilde{u}$, we have
\[
\big\|\tilde{U}^{i j} D_{i j} \tilde{w}\big\|_{L^{\infty}(\Omega)} \leq C_2.
\]
Let $K(x)$ be the Gauss curvature at $x \in \partial \Omega$. Then, since $\Omega$ is uniformly convex, we have \begin{equation}\label{Kom} 0<C^{-1}(\Omega)\leq K(x)\leq C(\Omega) \quad\text{on }\p\Omega.\end{equation} 
From the estimate (4.10) in the proof of  \cite[Lemma 4.2]{LeCPAM} with $\theta=0$ and $f_\delta:= -\Delta u +\textbf{b}\cdot Du+ c$ which uses (i) and (ii), we obtain
\begin{equation}\label{eq:K-u_nu-high}
\begin{aligned}
\int_{\partial \Omega} K \psi u_{\nu}^{n} d S \leq & \int_{\Omega}\left(\Delta u-\textbf{b}\cdot Du-c\right)(u-\tilde{u}) d x+C_3\left(\int_{\partial \Omega}\left(u_{\nu}^{+}\right)^{n} d S\right)^{(n-1)/ n}+C_3,
\end{aligned}
\end{equation}
where $C_3$ depends on $C_2$, $\Omega$ and $\varphi$.

We will estimate the first term on the right-hand side of \eqref{eq:K-u_nu-high} by splitting it into three terms. Firstly, using $u\Delta u=\div (u Du)-|Du|^2$ and integrating by parts, we have
\begin{align}
    \int_{\Omega}\Delta u(u-\tilde{u})\,dx&\leq\int_{\Omega}u\Delta u\,dx+C_2\int_{\Omega}\Delta u\,dx\nonumber\\
    &=\int_{\partial\Omega}\varphi u_{\nu}\,dS-\int_{\Omega}|Du|^2\,dx+C_2\int_{\partial\Omega}u_{\nu}\,dS\nonumber\\
    &\leq C(\varphi, C_2)\int_{\partial\Omega}|u_{\nu}|\,dS-\int_{\Omega}|Du|^2\,dx\nonumber\\
    &\leq C_4(n,\varphi, C_2)\left(\int_{\partial\Omega}(u^{+}_{\nu})^n\,dS\right)^{\frac{1}{n}} + C_4(n,\varphi, C_2)\quad(\text{recalling } (\ref{unuab})).\label{eq:1st-RHS}
\end{align}
Secondly, by integration by parts, we find
\begin{align}
\int_{\Omega}|Du|^2\,dx&=\int_{\Omega}(\div(uDu)-u\Delta u)\,dx\nonumber\\
&=\int_{\partial\Omega}\varphi u_{\nu}\,dS-\int_{\Omega}u\Delta u\,dx\nonumber\\
&\leq C_5(\varphi)\int_{\partial\Omega}u_{\nu}^+\,dS+\|u\|_{L^\infty(\Omega)}\int_{\Omega}\Delta u\,dx + C_5(\varphi)\nonumber\\
&\leq(C_5+\|u\|_{L^\infty(\Omega)})\int_{\partial\Omega}u_{\nu}^+\,dS + C_5.\label{eq:Du-L2bound-high}
\end{align}
In view of  \eqref{eq:Du-L2bound-high} with \eqref{eq:u-bound-high}, we can estimate
\begin{align}
    \int_{\Omega} \bb\cdot Du(\tilde{u}-u)\,dx &\leq |\Omega|^{1/2}\|\bb\|_{L^\infty(\Omega)}(\|\tilde{u}\|_{L^\infty(\Omega)}+\|u\|_{L^\infty(\Omega)})\left(\int_{\Omega}|Du|^2\,dx\right)^{\frac{1}{2}}\nonumber\\
    &\leq C_6+C_6\left(\int_{\partial \Omega}\left(u_{\nu}^{+}\right)^{n} d S\right)^{\frac{2}{n}}\label{eq:2nd-RHS}
\end{align}
where $C_6$ depends on $\Omega, n$, $\varphi$ and $\|\bb\|_{L^\infty(\Omega)}$. 
Moreover, $C_6$ depends linearly on $\|\bb\|_{L^\infty(\Omega)}$.

Finally, using \eqref{c-cond} and (\ref{eq:u-bound-high}), we have
\begin{align}
    \int_{\Omega}-c(u-\tilde{u})\,dx
    &\leq(\|u\|_{L^\infty(\Omega)}+\|\tilde{u}\|_{L^\infty(\Omega)})\int_{\Omega}|g_1|+|g_2||u|^{m}\,dx\nonumber\\
    &\leq C+C\|u\|_{L^\infty(\Omega)}^{m+1}\nonumber\\
    &\leq C_7+C_7\left(\int_{\partial\Omega}(u_{\nu}^+)^n\,dS\right)^{\frac{m+1}{n}}.\label{eq:3rd-RHS}
\end{align}
Here $C_7$ depends on $\Omega$, $n$, $\varphi$, $\|g_1\|_{L^1(\Omega)}$, $\|g_2\|_{L^1(\Omega)}$ and $m$.

It follows from (\ref{unuab}) that 
\begin{equation}
\label{unu+}
\int_{\partial\Omega}K\psi (u_{\nu}^+)^n\,dS\leq C_8(\Omega,\varphi,\psi) + \int_{\partial\Omega}K \psi u_{\nu}^{n} \,d S.
\end{equation}
Combining (\ref{Kom})--\eqref{eq:1st-RHS}, \eqref{eq:2nd-RHS}--(\ref{unu+}) while recalling that $0\leq m<n-1$ and $n\geq 3$, we obtain 
\begin{align*}
C^{-1}(\Omega)\min_{\partial\Omega}\psi\int_{\partial\Omega} (u_{\nu}^+)^n\,dS&\leq C_8 + \int_{\partial\Omega}K \psi u_{\nu}^{n} \,d S\\
    &\leq C_9\left[1+\left(\int_{\partial\Omega}(u_{\nu}^+)^n\,dS\right)^{\frac{n-1}{n}}+\left(\int_{\partial\Omega}(u_{\nu}^+)^n\,dS\right)^{\frac{m+1}{n}}\right],
\end{align*}
where $C_9$ depends on $C_3, C_4$, $C_6$, $C_7$ and $C_8$.
It follows that
\[
\int_{\partial\Omega}(u_{\nu}^+)^n\,dS\leq C
\]
where $C$ depends on $\Omega$, $n$, $\varphi$, $\psi$, $\|\bb\|_{L^\infty(\Omega)}$, $\|g_1\|_{L^1(\Omega)}$, $\|g_2\|_{L^1(\Omega)}$ and $m$.
The proof of the lemma is completed.
\end{proof}

\begin{rem} We have the following observations regarding the two dimensional version of Lemma \ref{lem:bound-u-high}.
\label{lem51_rem}
\begin{enumerate} 
\item[(i)] The above proof fails in two dimensions. This is because,  in two dimensions, the right-hand side of  \eqref{eq:2nd-RHS} is of the same order of magnitude as the left-hand side of  (\ref{eq:K-u_nu-high}). Therefore, when $C_6$ is large, plugging \eqref{eq:2nd-RHS} into (\ref{eq:K-u_nu-high}) does not give any new information.
\item[(ii)] On the other hand, since $C_6$ depends linearly on $\|\bb\|_{L^\infty(\Omega)}$, in two dimensions, one can still absorb the right-hand side of  \eqref{eq:2nd-RHS} into the left-hand side of  (\ref{eq:K-u_nu-high}) as long as $\|\bb\|_{L^\infty(\Omega)}$ is small, depending on $\Omega,\varphi$ and $\psi$. In this case, we still have the $L^\infty$ estimate. 
\item[(iii)]In Section \ref{SBV-2-pf}, we will establish the $L^\infty$ estimate in two dimensions under a stronger condition on
$\bb$ but $\|\bb\|_{L^\infty(\Omega)}$ can be arbitrarily large.
\end{enumerate}
\end{rem}

Next, we establish the Hessian determinant estimates.
\begin{lem}[Hessian determinant estimates]
\label{upper_lower}
Let $u\in W^{4,p}(\Omega)$ be a uniformly convex solution to the fourth order equation
\begin{equation}
\label{eq_lower}
  \left\{ 
  \begin{alignedat}{2}\sum_{i, j=1}^{n}U^{ij}D_{ij}w~& =-\Delta u+\bb\cdot D u+c(x)~&&\text{\ in} ~\ \ \Omega, \\\
 w~&= (\det D^2 u)^{-1}~&&\text{\ in}~\ \ \Omega,\\\
u ~&=\varphi~&&\text{\ on}~\ \ \p \Omega,\\\
w ~&= \psi~&&\text{\ on}~\ \ \p \Omega,
\end{alignedat}
\right.
\end{equation}
where $(U^{ij})=(\det D^2 u)(D^2 u)^{-1}$, $\min_{\p\Omega}\psi>0$, $\bb\in L^\infty(\Omega;\R^n)$ and $c\in L^{p}(\Omega)$ with $p>2n$.  Then there exists a constant $C>0$ depending on $\Omega$, $n$, $p$, $\varphi$, $\psi$, $\|\bb\|_{L^\infty(\Omega)}$ and $\|c\|_{L^p(\Omega)}$ such that
$$0<C^{-1}\leq \det D^2u\leq C\quad\text{in }\Omega.$$
\end{lem}
\begin{proof}
The proof uses a trick in Chau-Weinkove \cite{CW}. For simplicity,
denote \[d:=\det D^2u\quad\text{and } (u^{ij})= (D^2 u)^{-1}.\] Let \[G=d e^{Mu^2},\] where $M>0$ is a large constant to be determined later. By Lemma \ref{lem:bound-u-high}, 
we have \[\|u\|_{L^\infty(\Omega)}\leq C_0\]
where $C_0>0$ depends on $\Omega$, $n$, $\varphi$, $\psi$, $\|\bb\|_{L^\infty(\Omega)}$, and $\|c\|_{L^1(\Omega)}$.

 Since $w=d^{-1}$, we have
$w=G^{-1}e^{Mu^2}$. Direct calculations yield
\begin{eqnarray*}
  D_iw &=& -G^{-2}D_iGe^{Mu^2}+2MuD_iuG^{-1}e^{Mu^2}, \\[5pt]
  D_{ij} w&=& 2G^{-3}D_iGD_jGe^{Mu^2}-G^{-2}D_{ij}Ge^{Mu^2}\\[5pt] && -2MuD_juD_iGG^{-2}e^{Mu^2}-2MuD_iuD_jGG^{-2}e^{Mu^2}\\[5pt]
  &&+2MD_iuD_juG^{-1}e^{Mu^2}+2MuD_{ij}uG^{-1}e^{Mu^2}+4M^2u^2D_iuD_juG^{-1}e^{Mu^2}.
\end{eqnarray*}
Then, using $U^{ij} G^{-1} e^{Mu^2}= u^{ij}$,  we have
\begin{eqnarray*}
U^{ij}D_{ij} w&=&2G^{-2}u^{ij}D_iGD_jG-G^{-1}u^{ij}D_{ij}G-4MuG^{-1}u^{ij}D_iuD_jG\\[5pt]
&&+2Mu^{ij}D_iuD_ju+2Mnu+4M^2u^2u^{ij}D_iuD_ju\\[5pt]
&=&G^{-2}u^{ij}D_iGD_jG+u^{ij}(2MuD_iu-G^{-1}D_iG)(2MuD_ju-G^{-1}D_jG)\\[5pt]
&&-G^{-1}u^{ij}D_{ij}G+2Mu^{ij}D_iuD_ju+2Mnu\\[5pt]
&\geq&-G^{-1}u^{ij}D_{ij}G+2Mu^{ij}D_iuD_ju+2Mnu.
\end{eqnarray*}
Thus, from the first equation in \eqref{eq_lower}, we obtain
\[G^{-1}u^{ij}D_{ij}G\geq 2Mu^{ij}D_iuD_ju+2Mnu+\Delta u-\bb\cdot Du-c.\]
Using the following matrix inequality (see, for example, \cite[Lemma 2.8(c)]{LeCCM})
\[u^{ij} D_i v D_j v \geq \frac{|Dv|^2}{\trace(D^2u)} =\frac{|Dv|^2}{\Delta u}\]
together with $\Delta u\geq n d^{1/n}$, we find that
\begin{eqnarray}
\label{CWABP}
G^{-1}u^{ij}D_{ij}G
&\geq& 2M\frac{|Du|^2}{\Delta u}+\frac{1}{2}\Delta u-\bb\cdot Du+\frac{1}{2}\Delta u+2Mnu-c\nonumber\\[5pt]
&\geq&2\sqrt{M}|Du|-|\bb|\cdot |Du|+\frac{n}{2}d^{\frac{1}{n}}+2Mnu-c\nonumber\\[5pt]
&\geq&-(c-2Mnu-\frac{n}{2}d^{\frac{1}{n}})^+ \quad\text{in }\Omega,
\end{eqnarray}
provided \[M\geq \frac{1}{4}\|\bb\|_{L^\infty(\Omega)}^2.\] Hence, by the ABP estimate applied to (\ref{CWABP}) in $\Omega$ where $G= \psi^{-1}e^{M\varphi^2}$ on $\p\Omega$, we have
\begin{eqnarray}
\sup_{\Omega}G&\leq&\sup_{\partial\Omega}(\psi^{-1}e^{M\varphi^2})
+C(n,\Omega)\left\|\frac{(c-2Mnu-\frac{n}{2}d^{\frac{1}{n}})^+}{\big[\det (G^{-1} (D^2 u)^{-1})\big]^{\frac{1}{n}}}\right\|_{L^n(\Omega)}\nonumber\\[5pt]
&=&\sup_{\partial\Omega}(\psi^{-1}e^{M\varphi^2})
+C(n,\Omega)\left\|\frac{de^{Mu^2}(c-2Mnu-\frac{n}{2}d^{\frac{1}{n}})^+}{d^{-\frac{1}{n}}}\right\|_{L^n(\Omega)}\nonumber\\[5pt]
&\leq&\sup_{\partial\Omega}(\psi^{-1}e^{M\varphi^2})
+C_1\left\|d^{1+\frac{1}{n}}(c-2Mnu-\frac{n}{2}d^{\frac{1}{n}})^+\right\|_{L^n(\Omega)}.\label{ABP_lower}
\end{eqnarray}
Here $C_1$ depends on $n$, $\Omega$ and $C_0$ (via $\|u\|_{L^\infty(\Omega)}$). 
Note that, for any $p_0>n$, we have
\begin{multline}
\label{CWineq}
\left\|d^{1+\frac{1}{n}}(c-2Mnu-\frac{n}{2}d^{\frac{1}{n}})^+\right\|_{L^n(\Omega)}\leq\left(\int_{\big\{c-2Mnu\geq (n/2)d^{\frac{1}{n}}\big\}}d^{n+1}(c-2Mnu)^n\,dx\right)^{\frac{1}{n}}\\
\leq\left(\int_{\big\{ c-2Mnu\geq (n/2)d^{\frac{1}{n}}\big\}}d^{n+1}(c-2Mnu)^n\frac{(c-2Mnu)^{p_0-n}}{\big[(n/2)d^{1/n}\big]^{p_0-n}}\,dx\right)^{\frac{1}{n}}\\
=(n/2)^{-\frac{p_0-n}{n}}\left(\int_{\big\{c-2Mnu\geq (n/2)d^{\frac{1}{n}}\big\}}d^{n-\frac{p_0}{n}+2}(c-2Mnu)^{p_0}\,dx\right)^{\frac{1}{n}}.
\end{multline}
We now choose $p_0$ such that \[2n<p_0<\min\{n(n+2), p\}.\] 
Let $\gamma=1-\frac{p_0}{n^2}+\frac{2}{n}$. 
Then $0<\gamma<1$. Moreover, from (\ref{ABP_lower}) and (\ref{CWineq}), we have
\begin{eqnarray*}
\sup_{\Omega}G&\leq&C+C\left(\int_\Omega d^{n\gamma} (|c|+|u|)^{p_0} \,dx\right)^\frac{1}{n}\\[4pt]
&\leq &C+C\left(\int_\Omega (de^{Mu^2})^{n\gamma} (|c|^p+1) \,dx\right)^\frac{1}{n}\\[4pt]
&\leq &C_2+C_2(\sup_{\Omega}G)^\gamma \cdot \left(\int_\Omega  (|c|^p+1) \,dx\right)^\frac{1}{n}.
\end{eqnarray*}
Here $C_2$ depends on $\Omega, M,\varphi,\psi$, $C_1$, $\gamma$ and $p$. 
It follows that 
\[\sup_{\Omega}G\leq C_3(C_2, \gamma, \|c\|_{L^{p}(\Omega)}).\]
 Since $G= d e^{Mu^2}$, we also get an upper bound for $d=\det D^2 u$:
 \[\det D^2 u\leq C_3\quad\text{in }\Omega.\]
It remains to establish a positive lower bound for $\det D^2u$.

 Once we have the upper bound of the Hessian determinant of $u$, using $u=\varphi$ on  $\partial\Omega$ and a suitable barrier, we obtain 
\[\sup_{\Omega}|u| + \sup_\Omega|Du|\leq C_4(C_3,\varphi,\Omega).\]
Then we can apply the Legendre transform to get the lower bound of the determinant.
According to Proposition \ref{new2}, the Legendre transform $u^*$ of $u$ satisfies
 \[
 u^{\ast ij}D_{ij} \left(w^*+\frac{|y|^2}{2}\right)=\bb(D u^{*})\cdot y+ c(Du^{*}) \quad\text{in }\Omega^*=Du(\Omega),
 \]
where $u^{\ast ij}= (D^2 u^{\ast})^{-1}$ and $w^*=\log\det D^2u^*$. Applying the ABP estimate to $w^*+\frac{|y|^2}{2}$ on $\Omega^*$, and then changing of variables $y=Du(x)$ with $dy=\det D^2u \,dx$, we obtain
\begin{eqnarray*}
\sup_{\Omega^*}\Big(w^*+\frac{|y|^2}{2}\Big)&\leq& \sup_{\partial\Omega^*}\Big(w^*+\frac{|y|^2}{2}\Big)+C(n) \text{diam}(\Omega^*)
\left\|\frac{\bb(D u^{*}) \cdot y+ c(Du^{*})}{(\det u^{*ij})^{\frac{1}{n}}}\right\|_{L^n(\Omega*)}\\
&\leq&C(\psi, C_4)+C(n, C_4)\left(\int_{\Omega^*}\frac{|\bb(D u^{*}) \cdot y+ c(Du^{*})|^n}{(\det D^2u^*)^{-1}}\,dy\right)^{\frac{1}{n}}\\
&=&C(\psi, C_4)+C(n, C_4) \left(\int_{\Omega}|\bb \cdot Du+ c(x)|^n\,dy\right)^{\frac{1}{n}}\\[5pt]
&\leq& C(\psi, C_4)+C(n, C_4)(\|\bb\|_{L^n(\Omega)}\sup_\Omega |Du|+\|c\|_{L^n(\Omega)}).
\end{eqnarray*}
In particular, we have
\[\sup_{\Omega^*}w^* \leq C_5\]
where $C_5>0$ depending on $\Omega$, $n$, $\varphi$, $\psi$, $\|\bb\|_{L^\infty(\Omega)}$ and $\|c\|_{L^p(\Omega)}$.
Since $w^*=\log\det D^2u^*$, the above estimate gives the lower bound for $\det D^2u$:
\[\det D^2 u\geq e^{-C_5} \quad\text{in }\Omega,\]
completing the proof of the lemma.
\end{proof}

\begin{rem}
If there is no first order term, $\textbf{b}\cdot Du$ on the right-hand of \eqref{Abreu-lap}, we can directly obtain Hessian determinant bounds by the same trick used in the proof of Lemma \ref{upper_lower} without getting \textit{a priori} $L^\infty$-bound of $u$. Moreover, these bounds are valid for all dimensions.
\end{rem}
We are now ready to prove Theorem \ref{SBV-1}.
\begin{proof}[Proof of Theorem \ref{SBV-1}] The proof uses a priori estimates and degree theory as in that of Theorem \ref{thm:SBV-p-lap}. 
We obtain the existence of a uniformly convex solution in $C^{4,\alpha}(\overline{\Omega})$ in case $(i)$, and in $W^{4, p}(\Omega)$ in case $(ii)$,  with stated estimates provided that we can establish these estimates for $W^{4,p}(\Omega)$ solutions.  Thus, it remains to establish these a priori estimates.

Assume now $u\in W^{4,p}(\Omega)$ is a uniformly convex smooth solution to \eqref{Abreu-lap}. By Lemma \ref{lem:bound-u-high} and the assumption on $c$ in either $(i)$ or $(ii)$, we can obtain the Hessian determinant estimates for $u$ by Lemma \ref{upper_lower}.
Once we have the Hessian determinant estimates,  Theorem \ref{wHolder_thm} applies
with
\[F(x)=|x|^2/2,\quad\text{and }g(x)= b(x)\cdot Du(x) + c(x).\]
 This gives the H\"older estimates  for $w$.   The rest of the proof of Theorem \ref{SBV-1}, which is concerned with global higher order derivative estimates, is similar to {\it Step 2} in the proof of Theorem \ref{thm:SBV-p-lap}$(i)$ and $(ii)$.
We omit the details.
\end{proof}

\begin{rem}
\label{SBV-1-rem}
 In two dimensions, when $\|\bb\|_{L^\infty(\Omega)}$ is small, depending on $\Omega, \psi$ and $\psi$, the conclusions of Theorem \ref{SBV-1} still hold. Indeed, in this case, by Remark \ref{lem51_rem}, we still have the $L^\infty$ estimate in Lemma \ref{lem:bound-u-high}. The proof of Theorem \ref{SBV-1} then follows.
\end{rem}


\section{Proof of Theorem \ref{SBV-2}}
\label{SBV-2-pf}

In this section, we will prove Theorem \ref{SBV-2}.  
As in the proof of Theorem  \ref{thm:SBV-p-lap}, it suffices to derive the \textit{a priori} estimates for $W^{4,p}(\Omega)$  solutions. Here, we recall that
\[p>2.\]
Theorem \ref{SBV-2} can be deduced from the following Theorem \ref{thm:W4p-est}.
\begin{thm}[A priori $W^{4,p}(\Omega)$ estimates for $W^{4,p}(\Omega)$  solutions]
\label{thm:W4p-est}
 Let $\Omega\subset\R^2$, $\varphi$, $\psi$, $\bb$ and $c$ be as in Theorem \ref{SBV-2}. Assume that $u\in W^{4,p}(\Omega)$ is a uniformly convex solution to \eqref{Abreu-lap}. Then 
\[
\|u\|_{W^{4,p}(\Omega)}\leq C,
\]
where $C>0$ is a constant depending on $\Omega$, $p$, $\varphi$, $\psi$, $\bb$ and $c$.
\end{thm}

The rest of this section is devoted to the proof of Theorem \ref{thm:W4p-est}.

We will first obtain the $L^\infty$-bound of $u$ and $L^2$-bound of $Du$. For this, the following Poincar\'e type inequality will be useful.

\begin{lem}[Poincar\'e type inequality on planar convex domains]
\label{lem:Poincare}
Let $\Omega\subset\R^2$ be an open, smooth, bounded and uniformly convex domain. Assume that $u\in C^1(\Omega)\cap C(\overline{\Omega})$ and $u|_{\partial\Omega}=\varphi$. Then
\[
\int_{\Omega}|u|^2\,dx\leq C(\varphi,\operatorname{diam}(\Omega))\|u\|_{L^\infty(\Omega)}+\frac{\operatorname{diam}(\Omega)^2}{16}\int_{\Omega}|Du|^2\,dx.
\]
\end{lem}
\begin{proof}
Note that for any one-variable function $f\in C^1(a,b)\cap C^0[a,b]$ where $a< b$, one has
\begin{equation}\label{eq:one-f}
    \int_{a}^b |f(x)|^2\,dx\leq(b-a)(|f(a)|+|f(b)|)\|f\|_{L^\infty(a, b)}+\frac{(b-a)^2}{8}\int_{a}^b |f'(x)|^2\,dx.
\end{equation}
Indeed, denoting $c:=\frac{a+b}{2}$, then using H\"older's inequality and Fubini's theorem, one obtains 
\begin{align}
    \int_a^c|f(x)|^2\,dx&=\int_a^cf(a)(2f(x)-f(a))\,dx+\int_a^c\left(\int_a^x f'(t)\,dt\right)^2\,dx\nonumber\\
    &\leq 2(c-a)|f(a)|\cdot\|f\|_{L^\infty(a, b)}-(c-a)f(a)^2+\int_a^c(x-a)\int_{a}^x |f'(t)|^2\,dtdx\nonumber\\
    &=2(c-a)|f(a)|\cdot\|f\|_{L^\infty(a, b)}-(c-a)f(a)^2+\int_a^c|f'(t)|^2\int_t^c(x-a)\,dxdt\nonumber\\
    &\leq 2(c-a)|f(a)|\cdot\|f\|_{L^\infty(a, b)}-(c-a)f(a)^2+\frac{(c-a)^2}{2}\int_{a}^c|f'(x)|^2\,dx\nonumber\\
    &\leq (b-a)|f(a)|\cdot\|f\|_{L^\infty(a, b)}+\frac{(b-a)^2}{8}\int_{a}^c|f'(x)|^2\,dx.\label{eq:f-ac}
\end{align}
Similarly, we have 
\begin{equation}\label{eq:f-cb}
\int_{c}^{b}|f(x)|^2\,dx\leq (b-a)|f(b)|\cdot\|f\|_{L^\infty(a, b)}+\frac{(b-a)^2}{8}\int_{c}^{b}|f'(x)|^2\,dx.
\end{equation}
Combining \eqref{eq:f-ac} with \eqref{eq:f-cb}, we obtain  \eqref{eq:one-f}.

Next, by the convexity of $\Omega$, we can assume that there are $c, d\in\mathbb R$, and one-variable functions $a(x_1)$, $b(x_1)$, 
such that 
$$\Omega=\{(x_1,x_2):c<x_1<d, a(x_1)<x_2<b(x_1)\}.$$ It is clear that 
$d-c\leq\operatorname{diam}(\Omega)$ and $b(x_1)-a(x_1)\leq\operatorname{diam}({\Omega})$. Then, by \eqref{eq:one-f} and $u=\varphi$ on $\p\Omega$, we have 
\begin{multline*}
    \int_{a(x_1)}^{b(x_1)}|u(x_1,x_2)|^2\,dx_2\leq 2\operatorname{diam}({\Omega})\|\varphi\|_{L^\infty(\Omega)}\|u\|_{L^\infty(\Omega)}\\+\frac{\operatorname{diam}({\Omega})^2}{8}\int_{a(x_1)}^{b(x_1)}|D_{x_2}u(x_1,x_2)|^2\,dx_2.
\end{multline*}
Integrating the above inequality over $c<x_1<d$ yields
\begin{equation}\label{eq:int-D2u}
    \int_{\Omega}|u|^2\,dx\leq 2\operatorname{diam}({\Omega})^2\|\varphi\|_{L^\infty(\Omega)}\|u\|_{L^\infty(\Omega)}+\frac{\operatorname{diam}({\Omega})^2}{8}\int_{\Omega}|D_{x_2}u|^2\,dx.
\end{equation}
Similarly,
\begin{equation}\label{eq:int-D1u}
    \int_{\Omega}|u|^2\,dx\leq 2\operatorname{diam}({\Omega})^2\|\varphi\|_{L^\infty(\Omega)}\|u\|_{L^\infty(\Omega)}+\frac{\operatorname{diam}({\Omega})^2}{8}\int_{\Omega}|D_{x_1}u|^2\,dx.
\end{equation}
Combining \eqref{eq:int-D2u} and \eqref{eq:int-D1u}, we obtain
\[
    \int_{\Omega}|u|^2\,dx\leq 2\operatorname{diam}({\Omega})^2\|\varphi\|_{L^\infty(\Omega)}\|u\|_{L^\infty(\Omega)}+\frac{\operatorname{diam}({\Omega})^2}{16}\int_{\Omega}|Du|^2\,dx,
\]
completing the proof of the lemma.
\end{proof}

\subsection{Estimates for $\sup_\Omega |u|$ and $\|Du\|_{L^2(\Omega)}$}
Now we derive bounds for $u$ and $\|Du\|_{L^2(\Omega)}$.

\begin{lem}[$L^\infty$ and $W^{1,2}$ estimates]\label{lem:bound-Du}
Let $\Omega\subset\R^2$, $\varphi$, $\psi$, $\textbf{b}$ and $c$ be as in Theorem \ref{SBV-2}. Assume that $u\in W^{4,p}(\Omega)$ is a uniformly convex solution to \eqref{Abreu-lap}. Then there exists a constant $C>0$ depending on $\Omega$, $\varphi$, $\psi$, $\bb$ and $\|c\|_{L^1(\Omega)}$ such that
\[
\|u\|_{L^\infty(\Omega)}\leq C\quad\text{and}\quad\|Du\|_{L^2(\Omega)}\leq C.
\]
\end{lem}

\begin{proof}
To prove the lemma where $n=2$, by \eqref{eq:u-bound-high} and \eqref{eq:Du-L2bound-high}, it suffices to prove
\begin{equation}
\label{unuL2}
\int_{\partial \Omega}u_{\nu}^{2} d S \leq C(\Omega, \varphi, \psi, \bb, \|c\|_{L^1(\Omega)}),
\end{equation}
where $\nu$ is the unit outer normal of $\p\Omega$. 

Let $\tilde u$ be as in the proof of Lemma \ref{lem:bound-u-high} so that $(i)$ and $(ii)$ there are satisfied. Let $K(x)$ be the Gauss curvature at $x \in \partial \Omega$. Then, as in  (\ref{eq:K-u_nu-high}), we have, for some $C_1(\Omega,\varphi)>0$
\begin{equation}\label{eq:K-u_nu}
\begin{aligned}
\int_{\partial \Omega} K \psi u_{\nu}^{2} d S \leq & \int_{\Omega}\left(\Delta u-\bb\cdot Du-c\right)(u-\tilde{u}) d x+C_1\Big(\int_{\partial \Omega}u_{\nu}^{2} d S\Big)^{1 / 2}+C_1
\end{aligned}
\end{equation}
Next, we will estimate the RHS of \eqref{eq:K-u_nu} term by term. First, from the inequality before last in \eqref{eq:1st-RHS}, we have
\begin{equation}
    \int_{\Omega}\Delta u(u-\tilde{u})\,dx\leq
    C(\Omega,\varphi)\left(\int_{\partial\Omega}u_{\nu}^2\,dS\right)^{\frac{1}{2}}-\int_{\Omega}|Du|^2\,dx.\label{eq:RHS-1st}
\end{equation}
Using $u=\varphi$ on $\p\Omega$, and integrating by parts, we get
\begin{align}
    \int_{\Omega}(\bb\cdot Du)\tilde{u}\,dx&=  \int_{\Omega}(\bb \tilde u)\cdot Du\,dx\nonumber\\
    &=\int_{\partial\Omega}u\tilde{u}\bb\cdot \nu\,dS-\int_{\Omega}\div(\bb\tilde{u})u\,dx\nonumber\\
    &=\int_{\partial\Omega}\varphi\tilde{u}(\bb\cdot\nu)\,dS-\int_{\Omega}(\bb\cdot D\tilde{u}+\tilde{u}\,\div \bb) u\,dx\nonumber\\&\leq C(1+ \|u\|_{L^{\infty}(\Omega)})\leq C_3 + C_3 \Big(\int_{\partial \Omega}u_{\nu}^{2} d S\Big)^{1 / 2},\label{eq:RHS-2nd-1}
\end{align}
where $C_3$ depends on $\Omega$, $\varphi$, $\sup_{\p\Omega}|\bb|$, $\|\bb\|_{L^{\infty}(\Omega)}$ and $\|\div\bb\|_{L^{\infty}(\Omega)}$.

Moreover,
\begin{align*}
    \int_{\Omega}-(\bb\cdot Du)u\,dx =\frac{1}{2}\int_{\Omega}-\bb\cdot D(u^2)\,dx&=\frac{1}{2}\left[\int_{\Omega}(\div \bb)u^2\,dx-\int_{\partial\Omega}u^2\bb\cdot\nu\,dS\right].
\end{align*}
Note that $\div \bb\leq \frac{32}{\operatorname{diam}(\Omega)^2}$. Then by Lemma \ref{lem:Poincare} and \eqref{eq:u-bound-high}, we have
\begin{align*}
\frac{1}{2}\int_{\Omega}(\div \bb)u^2\,dx&\leq C(\varphi,\operatorname{diam}(\Omega))\|u\|_{L^\infty(\Omega)}+\int_{\Omega}|Du|^2\,dx\\
&\leq C(\Omega,\varphi)+C(\Omega,\varphi)\left(\int_{\partial\Omega}u_{\nu}^2\,dS\right)^{\frac{1}{2}}+\int_{\Omega}|Du|^2\,dx.
\end{align*}
Hence
\begin{align}
    \int_{\Omega}-(\bb\cdot Du)u\,dx &=\frac{1}{2}\left[\int_{\Omega}(\div \bb)u^2\,dx-\int_{\partial\Omega}\varphi^2\bb\cdot\nu\,dS\right]\nonumber\\
    &\leq C_4+C_4\left(\int_{\partial\Omega}u_{\nu}^2\,dS\right)^{\frac{1}{2}}+\int_{\Omega}|Du|^2\,dx,\label{eq:RHS-2nd-2}
\end{align}
where $C_4$ depends on $\Omega$, $\varphi$, and $\sup_{\p\Omega}|\bb|$.

Finally, as in  \eqref{eq:3rd-RHS}, we get
\begin{equation}
    \int_{\Omega}-c(u-\tilde{u})\,dx\leq C_5+C_5\left(\int_{\partial\Omega}u_{\nu}^2\,dS\right)^{\frac{1}{2}},\label{eq:RHS-3rd}
\end{equation}
where $C_5$ depends on $\Omega,\varphi$ and $\|c\|_{L^1(\Omega)}$.

Combining \eqref{eq:K-u_nu}--\eqref{eq:RHS-3rd}, we obtain
\begin{align*}
    C^{-1}(\Omega)\inf_{\partial\Omega}\psi\int_{\partial\Omega}u_{\nu}^2\,dS&\leq\int_{\partial\Omega}K \psi u_{\nu}^{2} \,d S\leq C_6\left[1+\left(\int_{\partial\Omega}u_{\nu}^2\,dS\right)^{\frac{1}{2}}\right],
\end{align*}
where $C_6>0$ depends on $\Omega$, $\varphi$, $\psi$, $\bb$ and $\|c\|_{L^1(\Omega)}$. From this, we deduce (\ref{unuL2}), completing 
the proof of the lemma.
\end{proof}

\subsection{Hessian determinant estimates for $u$}

\begin{lem}[Hessian determinant estimates]\label{det-est-dim2}
 Let $\Omega\subset\R^2$, $\varphi$, $\psi$, $\bb$ and $c$ be as in Theorem \ref{SBV-2}. Assume that $u\in W^{4,p}(\Omega)$ is a uniformly convex solution to \eqref{Abreu-lap}. Then 
\[
0<C^{-1}\leq \det D^2u\leq C \quad\text{in }\Omega,
\]
where $C>0$ is a constant depending on $\Omega$, $\varphi$, $\psi$, $\bb$ and $\|c\|_{L^2(\Omega)}$.
\end{lem}
\begin{proof}
We first prove the lower bound of $\det D^2u$. Note that in two dimensions, we have $\text{trace } U=\Delta u$. Hence we can rewrite the first equation in \eqref{Abreu-lap} as
\begin{equation}\label{eq:re-Abreu}
    U^{ij}D_{ij}\Big(w+|x|^2/2\Big)=\bb(x)\cdot Du(x)+c(x):=Q(x)\quad\text{in }\Omega.
\end{equation}
By Lemma \ref{lem:bound-Du}, we have \[\|Q\|_{L^2(\Omega)}\leq C_0\]
where $C_0$ depends on $\Omega$, $\varphi$, $\psi$, $\bb$ and $\|c\|_{L^2(\Omega)}$.
 
 Applying the ABP estimate to \eqref{eq:re-Abreu} and using $\det U= \det D^2 u$, we have 
\begin{eqnarray*}
   \sup_{\Omega}\left(w+|x|^2/2\right)&\leq& \sup_{\partial\Omega}\psi+C(\Omega)+C(\Omega)\left\|\frac{Q}{(\det U)^{1/2}}\right\|_{L^2(\Omega)}\\
    &\leq& C(\Omega,\psi)+C(\Omega)\|Q\|_{L^2(\Omega)}\cdot\sup_{\Omega}(\det D^2u)^{-\frac{1}{2}}\\
    &\leq& C(\Omega,\psi)+C(\Omega)(\sup_{\Omega}w)^{\frac{1}{2}}.
\end{eqnarray*}
Therefore
$\sup_{\Omega}w\leq C_1$,
where  $C_1$ depends on $\Omega$, $\varphi$, $\psi$, $\bb$ and $\|c\|_{L^2(\Omega)}$. Consequently,
\begin{equation}\label{eq:lower-bound-d}
\det D^2u\geq C_1^{-1}>0\quad\text{in }\Omega.
\end{equation}
Hence by the boundary H\"older estimate for solutions of non-uniformly elliptic equations \cite[Proposition 2.1]{Le1}, we know from (\ref{eq:re-Abreu}) that $w$ is H\"older continuous on $\partial\Omega$ with estimates depending only on $C_1$, $\Omega$ and $\psi$. Then by constructing a suitable barrier near the boundary as in \cite[Lemma 2.5]{Le2}, we can obtain 
\[\|Du\|_{L^\infty(\Omega)}\leq C_2,\]
where $C_2$ depends on $C_1$, $\Omega$, $\varphi$ and $\psi$. 

The  upper bound of the Hessian determinant can be obtained similar as in Lemma \ref{upper_lower}. Let $u^\ast(y)$ be the Legendre transform of $u(x)$ where \[y= Du(x)\in Du(\Omega):=\Omega^*.\] Then
\[\text{diam}(\Omega^*)\leq C_2.\]
By Proposition \ref{new2} (with $F(x)=|x|^2/2$), $u^\ast$ satisfies
\begin{equation}\label{eq:dual-eq}
U^{*ij}D_{ij}\Big(-w^{*}-|y|^2/2\Big)=-Q(Du^*)\det D^2u^*\quad\text{in }\Omega^\ast
\end{equation}
where $(U^{\ast ij})= (\det D^2 u^*) (D^2 u^*)^{-1}$, and $w^*=\log\det D^2u^*$. 

Applying the ABP maximum principle to \eqref{eq:dual-eq}, and recalling that \[w^\ast(y) = \log (\det D^2 u(x))^{-1}=\log w(x) =\log \psi(x)\quad\text{on } \p\Omega^\ast,\] we obtain
\begin{align*}
\sup_{\Omega^*}(-w^*-|y|^2/2)&\leq \sup_{\p\Omega^*}(-w^*-|y|^2/2) +C(\text{diam}(\Omega^*))\|Q(Du^*)(\det D^2u^*)^{1/2}\|_{L^2(\Omega^*)}\\
&\leq-\log \min_{\partial\Omega}\psi+C(C_2)\|Q\|_{L^2(\Omega)},
\end{align*}
where we used
\[
\int_{\Omega^*}[Q(Du^*)]^2\det D^2u^*\,dy=\int_{\Omega}[Q(x)]^2\det D^2u^*\det D^2u\,dx=\int_{\Omega}[Q(x)]^2\,dx= \|Q\|_{L^2(\Omega)}^2.
\]
Therefore, we have
\[\sup_{\Omega^*}(-w^*)\leq C_3\]
where $C_3$ depends on $C_0$, $C_2$ and $ \min_{\partial\Omega}\psi$. 
This implies $w^*\geq -C_3$ in $\Omega^*$, and hence
\begin{equation}\label{eq:upper-bound-d}
\det D^2u\leq e^{C_3}\quad\text{in }\Omega.
\end{equation}
The lemma follows from
(\ref{eq:lower-bound-d}) and (\ref{eq:upper-bound-d}).
\end{proof}

\subsection{Proof of Theorem \ref{thm:W4p-est}} Finally, we can prove Theorem  \ref{thm:W4p-est} which implies Theorem \ref{SBV-2}.
\begin{proof}[Proof of Theorem \ref{thm:W4p-est}]
Once we have the determinant estimates, we can establish the higher estimates by using the regularity of the linearized Monge-Amp\`ere equation with drift terms as in  Section \ref{SBV-plap-pf} and Section \ref{SBV-1-pf}.
However, in two dimensions, we can also establish these estimates as in \cite{LeCPAM}.

By Lemma \ref{det-est-dim2}, we have 
\beq\label{det22}
0<\lambda\leq \det D^2u\leq \Lambda \quad\text{in }\Omega
\eeq
 for $\lambda$, $\Lambda$ depending on $\Omega$, $\varphi$, $\psi$, $\textbf{b}$, $p$ and $\|c\|_{L^p(\Omega)}$. By the interior $W^{2,1+\varepsilon}$ estimates for Monge-Amp\`ere equation \cite{DFS,F, Sc}, we have $D^2 u\in L^{1+\e}_{loc}(\Omega)$ for some  constant $\e(\lambda,\Lambda)>0$.
By the global $W^{2,1+\varepsilon}$ estimates for the Monge-Amp\`ere equation \cite{S3}, there exists a constant $C_0>0$ depending on $\Omega$, $\varphi$, $\psi$, $\textbf{b}$, $p$ and $\|c\|_{L^p(\Omega)}$ such that
\[
\|u\|_{W^{2,1+\varepsilon(\lambda,\Lambda)}(\Omega)}\leq C_0.
\]
Let $q:=\min\{p,1+\varepsilon(\lambda,\Lambda)\}>1$. Then \[G:= -\Delta u+\bb\cdot Du+c\]  satisfies
\[\|G\|_{L^q(\Omega)} \leq C_1\]
where $C_1>0$ depending on $\Omega$, $p$, $\varphi$, $\psi$, $\textbf{b}$ and $\|c\|_{L^p(\Omega)}$. 
Recall that 
\[
U^{ij}D_{ij} w=G\quad\text{on }\Omega,\quad\text{and } w=\psi \quad\text{on }\p\Omega.
\]
By the global H\"older estimate for the linearized Monge-Amp\`ere equation \cite{LN2} with $L^{q}$ right-hand side where $q>n/2$, we deduce
\[
\|w\|_{C^\alpha(\overline{\Omega})}\leq C(\Omega, \varphi, \psi, p, \bb, c)
\]
where $\alpha\in (0, 1)$ depends on $\Omega, \varphi, \psi, p, \bb, c$. The proof of $W^{4, p}(\Omega)$ estimate for $u$ is now the same as that of
Theorem \ref{thm:SBV-p-lap}$(ii)$. Hence, the theorem is proved.
\end{proof}

\section{Extensions and the proof of Theorem \ref{LMA-G}}
\label{rem_sect}
In this section, we discuss (\ref{Abreu-sin})  with more general lower order terms,
and present a proof of Theorem \ref{LMA-G} for completeness. 
\subsection{Possible extensions of the main results}
The following remarks indicate some possible extensions of our main results.

\begin{rem}
From the proofs in Sections \ref{SBV-plap-pf}-\ref{SBV-2-pf} and the $L^\infty$-estimates in Lemma \ref{lem:bound-u-high}, it can be seen that some conclusions of Theorems \ref{thm:SBV-p-lap}, \ref{SBV-1} and \ref{SBV-2}
 also hold for more general cases of $c=c(x,z)$. Consider,  for example, 
\[c(x,z)=g_1(x)+g_2(x) h(z).\]
Then the following facts hold:  
\begin{enumerate}
\item[(1)] The conclusions in Theorem \ref{thm:SBV-p-lap}(ii) hold when $g_1\leq 0, g_2\leq 0; g_1, g_2\in L^p(\Omega)$ with $p>n$, and $h\geq 0$ with $h\in C^\alpha(\R)$.

\item[(2)] The conclusions in  Theorem \ref{SBV-1}(i) hold when $g_1, g_2\in C^\alpha(\overline\Omega)$, and $h\in C^\alpha(\R)$ with $|h(z)|\leq C |z|^{m}$ for $0\leq m<n-1$.
\item[(3)] The conclusions in  Theorem \ref{SBV-1}(ii) hold   when $g_1, g_2\in L^p(\Omega)$ with $p>2n$, and $h\in C^\alpha(\R)$ with $|h(z)|\leq C |z|^{m}$ for $0\leq m<n-1$.

\item[(4)] The conclusions in Theorem \ref{SBV-2} hold when  $g_1, g_2\in L^p(\Omega)$ with $p>2$, and  $h\in C^\alpha(\R)$ with $|h(z)|\leq C |z|^{m}$ for $0\leq m<1$.
\end{enumerate}
\end{rem}
\begin{rem}  Since we use the trace of $\bb$ on $\p\Omega$ in 
(\ref{eq:RHS-2nd-1}), it is natural to have $\bb\in C(\overline{\Omega};\R^n)$. It would be interesting to obtain the conclusion of Theorem \ref{SBV-2} for $\bb\in C(\overline{\Omega};\R^n)$ instead of $\bb\in C^{1}(\overline{\Omega};\R^n)$.
\end{rem}
\subsection{Global H\"older estimates for pointwise H\"older continuous solutions at the boundary}
\label{sect_LOT_proof}

In this section, we prove Theorem \ref{LMA-G}. 

The proof is similar to that of \cite[Theorem 1.4]{Le1} for the case without a drift. For completeness, we include the proof which includes the following ingredients: interior H\"older estimates for linearized Monge-Amp\`ere equations with bounded drifts, and rescalings using a consequence of the boundary Localization Theorem for the Monge-Amp\`ere equation which we will recall below.

 Under the assumption $\lambda\leq\det D^2 u\leq \Lambda$, the linearized Monge-Amp\`ere operator $U^{ij}D_{ij}$ is elliptic, but it can be degenerate and singular in the sense that the eigenvalues of $U=(U^{ij})$ can tend to zero or infinity. To prove estimates for the  linearized Monge-Amp\`ere equation that are independent of the bounds on the eigenvalues of $U$, as in \cite{CG} and subsequent works, we work with {\it sections} of $u$
instead with Euclidean balls. For a convex function $u\in C^{1}(\bom)$ defined on the closure of a convex, bounded domain $\Omega\subset \R^n$,
the section of $u$ centered at $x\in \overline \Omega$ with height $h>0$ is defined by
\begin{equation*}
 S_{u} (x, h) :=\Big\{y\in \overline \Omega: \quad u(y) < u(x) + Du(x) \cdot (y- x) +h\Big\}.
\end{equation*}

Before proving the global H\"older estimate, we recall the  interior H\"older estimate.
The following interior H\"older estimate for the 
nonhomogeneous linearized Monge-Amp\`ere equation with drift terms 
is a simple consequence of the interior Harnack inequality proved in \cite[Theorem 1.1]{LeCCM}. In \cite{M}, Maldonado proved a similar Harnack's inequality for linearized Monge-Amp\`ere equation with drift terms with different and stronger conditions on $\bb$.
\begin{thm}[Interior H\"older estimate for the 
nonhomogeneous linearized Monge-Amp\`ere equation with drift terms, {\cite{LeCCM}}]
\label{inter-H}
 Suppose that $u\in C^2(\Omega)$ is a strictly convex function in a bounded domain $\Omega\subset \R^n$ with 
  section $S_u(0, 1)$ satisfying 
 \[B_{r_1}(0)\subset S_{u}(0, 1)\subset B_{r_2}(0)\]
 for some positive constants $r_1\leq r_2$, and 
 its Hessian determinant satisfying
 \[ \lambda \leq \det D^2 u\leq\Lambda\quad\text{in }\Omega\]
 where $\lambda$ and $\Lambda$ are positive constants. Let $(U^{ij})= (\det D^2 u)(D^2 u)^{-1}$. Let $\bb: S_u(0, 1)\rightarrow\R^n$ be a vector field such that
$\|\bb\|_{L^{\infty}(S_{u}(0, 1))} \leq M$. 
 Let $v \in  W^{2,n}_{loc}(
S_{u}(0, 1))$  be a  solution to 
\[
U^{ij}D_{ij} v + \bb\cdot Dv = f\text{ in  }S_{u}(0, 1).\]
Then, there exist 
constants $\beta_0, C >0 $ depending only $\lambda, \Lambda$, $n$, $r_1$, $r_2$, and $M$
such that 
$$|v(x)-v(y)|\leq C|x-y|^{\beta_0}\Big(
\|v\|_{L^{\infty}( S_{u}(0, 1))}  + \|f\|_{L^{n}( S_{u}(0, 1))} \Big),~\quad\text{for all }x, y\in  S_{u}(0, 1/2). $$
\end{thm}

To bridge the interior H\"older estimates in Theorem \ref{inter-H} and the boundary H\"older estimates in (\ref{Bdr_assumed}), we need to control the shape of sections of the convex function $u$ that are tangent to the boundary $\p \Omega$. The following proposition, proved by Savin in \cite{S3} (see also \cite[Proposition 3.2]{LS}),  provides such a tool.
It is a consequence of the boundary Localization Theorem for the Monge-Amp\`ere equation, proved by Savin in \cite[Theorem 2.1]{S1} and \cite[Theorem 3.1]{S2}.

\begin{prop}[Shape of sections tangent to the boundary, {\cite{S3}}]\label{tan_sec}
Assume that $\Omega\subset\R^n$ is a uniformly convex domain with boundary $\p\Omega\in C^3$.
Let $u \in C(\overline 
\Omega) 
\cap 
C^2(\Omega)$ be a convex function satisfying 
\[\lambda \leq \det D^2 u  \leq \Lambda \quad \text{in}\quad \Omega\]
for some positive constants $\lambda$ and $\Lambda$. Moreover, assume that $u|_{\p\Omega}\in C^3$.
 Assume that for some $y \in \Omega$ the section $S_{u}(y, h) \subset \Omega$
is tangent to $\p \Omega$ at some point $x_0\in\p\Omega$, that is, $\p S_{u}(y, h)\cap
\p\Omega =x_0$, for some $h \le h_0(\lambda, \Lambda, \Omega, u|_{\p\Omega}, n)$. Then there exists a small positive
 constant $k_0$ depending on $\lambda$, $\Lambda$, $\Omega, u|_{\p\Omega} $ and $n$ such that
$$k_0 E_h \subset S_{u}(y, h) -y\subset k_0^{-1} E_h, \quad \quad k_0 h^{1/2} \leq \dist(y,\p \Omega) \leq k_0^{-1} h^{1/2}, \quad $$
where $E_h= h^{1/2}A_{h}^{-1}B_1(0)$ is an ellipsoid with $A_h$ being a linear transformation with the following properties  
\begin{equation*} \|A_{h}\|, \|A_{h}^{-1} \| \le k_0^{-1} |\log  h|;~ \det A_{h}=1.
\end{equation*}

\end{prop}

Now, we are ready to prove Theorem \ref{LMA-G}.

\begin{proof}[Proof of Theorem \ref{LMA-G}] 
By considering the equation satisfied by $(\|\varphi\|_{C^{\alpha}(\p\Omega)} + \|f\|_{L^{n}(\Omega)})^{-1}v$,
we can assume that
$$\|\varphi\|_{C^{\alpha}(\p\Omega)} + \|f\|_{L^{n}(\Omega)}=1,$$
and we need to show that 
$$\|v\|_{C^{\beta}(\overline{\Omega})}\leq C(\lambda, \Lambda, n, \alpha, \Omega, u|_{\p\Omega}, \gamma,\delta, K, M),$$
for some $\beta\in (0, 1)$ depending on $n,\lambda,\Lambda$, $\Omega, u|_{\p\Omega}$, $\gamma$, and $M$.

{\it Step 1: H\"older estimates in the interior of a section tangent to the boundary.}
Let $y\in \Omega $ with \[r=r_y:=\text{dist} (y,\partial\Omega) \le c_1(n,\lambda,\Lambda,\Omega, u|_{\p\Omega}),\] and consider the maximal interior section $S_u(y, h)$ centered at $y$, that is
\[h=h_y:=\sup\{t\,|~ S_u(y,t)\subset \Omega\}.\]

By Proposition \ref{tan_sec} applied at the point $x_0\in \p S_u(y, h) \cap \p \Omega,$ we can find a constant $k_0(n,\lambda,\Lambda,\Omega, u|_{\p\Omega})>0$ such that
 \begin{equation}
 \label{hrk0}
 k_0 h^{1/2} \leq r \leq k_0^{-1}  h^{1/2},
\end{equation}
and $S_u(y,h)$ is equivalent to an ellipsoid $E_h$, that is, 
$$k_0E_h \subset S_u(y,  h)-y \subset k_0^{-1}E_h,$$
where
\begin{equation}E_h:= h^{1/2}A_{h}^{-1}B_1(0),~ \text{ with } \|A_{h}\|, \|A_{h}^{-1} \| \le k_0^{-1} |\log  h|;~ \det A_{h}=1.
\label{Eh0}
\end{equation}
Let
$$T\tilde x:=y+h^{1/2}A_{h}^{-1}\tilde x.$$
We rescale $u$ by
$$\tilde u(\tilde x):=\frac {1}{  h} [u(T\tilde x)-u(y)-D u(y)\cdot (T\tilde x-y)].$$
Then
\[\lambda\leq \det D^2\tilde u(\tilde x)\leq\Lambda,\]
and
\begin{equation}
\label{normalsect0}
B_{k_0}(0) \subset \tilde S_1 \subset B_{k_0^{-1}}(0), \quad \quad \tilde S_1:= S_{\tilde u}(0, 1)=h^{-1/2} A_{ h}(S_u(y,  h)- y).
\end{equation}
Define the rescalings $\tilde v$ for $v$, $\tilde \bb $ for $\bb$, and $\tilde g$ for $g$ by
$$\tilde v(\tilde x):= v(T\tilde x)- v(x_{0}),\quad \tilde \bb (\tilde x) =  h^{1/2} A_{ h} \bb(T\tilde x), \quad  \tilde g(\tilde x):= h g(T\tilde x),\quad \tilde x\in \tilde S_{1}.$$
Simple computations give
\[D \tilde v(\tilde x)= h^{1/2}(A_{h}^{-1})^{t} D v(T\tilde x), \]
\begin{equation*}
 D^2 \tilde u(\tilde x)= (A_{h}^{-1})^{t} D^2 u(T\tilde x) A_{h}^{-1},\quad  D^2 \tilde v(\tilde x)= h(A_{h}^{-1})^{t} D^2 v(T\tilde x) A_{h}^{-1}, 
\end{equation*}
and the cofactor matrix $\tilde U=(\tilde U^{ij})$ of $D^2 \tilde u$ satisfies
$$\tilde U(\tilde x): = (\det D^2 \tilde u) (D^2 \tilde u)^{-1} = (\det D^2 u) A_{h} (D^2 u)^{-1} (A_{h})^{t} = A_h U(T\tilde x) (A_h)^{t}.$$
Therefore, we find that
\begin{equation*}\tilde U^{ij} D_{ij} \tilde v= \trace (\tilde U D^2 \tilde v)=h (U^{ij} D_{ij} v)(T\tilde x)~\text{ in }~\tilde S_1.
\end{equation*}
It is now easy to see that 
 $\tilde v$ solves
\[\tilde U^{ij} D_{ij}\tilde v + \tilde \bb \cdot D\tilde v= \tilde g\quad \text{in }\tilde S_{1}.\]
Due to (\ref{Eh0}), and the smallness of $h$ (see (\ref{hrk0})), we have the following bound
\begin{equation*}
\|\tilde \bb\|_{L^{\infty}(\tilde S_1)} \leq k_0^{-1}h^{1/2}|\log h|\cdot\|\bb\|_{L^{\infty}(S_u(y, h))} \leq k_0^{-1}h^{1/2}|\log h| M \leq M.
\end{equation*}
Now, we apply the interior H\"older estimates in Theorem \ref{inter-H} to $\tilde v $ to obtain a small constant $\beta\in (0,1)$ depending only on $n, \lambda, \Lambda$, $k_0$, and $M$, such that
\begin{multline*}\abs{\tilde v (\tilde z_{1})-\tilde v(\tilde z_{2})}\leq C_1(n, \lambda, \Lambda, M)\abs{\tilde z_{1}-\tilde z_{2}}^{\beta}\Big \{\|\tilde v \|_{L^{\infty}(\tilde S_{1})} + \|\tilde g\|_{L^{n}(\tilde S_{1})}\Big\},\\ \text{for all } \tilde z_{1}, \tilde z_{2}\in \tilde S_{1/2}:= S_{\tilde u}(0, 1/2).\end{multline*}
By (\ref{normalsect0}), we can decrease $\beta$ in the above inequality if necessary, and thus assume that
$$2\beta\leq \gamma.$$ 
A simple computation using (\ref{Eh0}) gives
$$ \|\tilde g\|_{L^{n}(\tilde S_{1})} = h^{1/2}\|g\|_{L^{n}(S_u(y, h))}.$$
Moreover, from (\ref{hrk0}) and (\ref{Eh0}), we infer the following inclusions regarding sections and balls
\begin{equation}
\label{ballsecth2}
B_{c_2 \frac{r}{\abs{\log r}}}(y)\subset S_u(y,h/2) \subset S_u(y,h) \subset B_{C_2 r\abs{\log r}}(y),
\end{equation}
for some $c_2\in (0, 1)$ and  $C_2>0$ depending on $n,\lambda,\Lambda,\Omega, u|_{\p\Omega}$. We also deduce that
$$\text{diam} (S_u(y,h))\leq C(n,\lambda,\Lambda,\Omega, u|_{\p\Omega})r\abs{\log r}\leq \delta$$
if 
\[r\leq c_3(n,\lambda,\Lambda, \Omega, u|_{\p\Omega}, \delta).\]
We now consider $r$ satisfying the above inequality. 
By (\ref{Bdr_assumed}), we have
$$\norm{\tilde v }_{L^{\infty}(\tilde S_{1})} \leq K \text{diam} (S_u(y, h))^{\gamma} \leq C_3 (r\abs{\log r})^{\gamma},$$
where $C_3= C_3(n,\lambda,\Lambda,\Omega, u|_{\p\Omega},\gamma, K)$.
Hence
\begin{equation*}\abs{\tilde v (\tilde z_{1})-\tilde v(\tilde z_{2})}\leq C_4\abs{\tilde z_{1}-\tilde z_{2}}^{\beta}\Big\{(r\abs{\log r})^{\gamma}  + h^{1/2}\|g\|_{L^{n}(S_u
(y, h))}\Big\} \quad \text{for all }\tilde z_{1}, \tilde z_{2}\in \tilde S_{1/2}\end{equation*}
where $C_4=C_4(n,\lambda,\Lambda,\Omega, u|_{\p\Omega},\delta, \gamma, K, M)$ 

Each $z\in S_u(y,  h/2)$ corresponds to a unique $\tilde z= T^{-1}z\in \tilde S_{1/2}$. 
 Rescaling back, recalling $2\beta\leq\gamma$, and using
$\tilde z_1-\tilde z_2= h^{-1/2}A_{h}(z_1-z_2)$,
and the fact that
\begin{eqnarray*}\abs{\tilde z_1-\tilde z_2}&\leq& \| h^{-1/2}A_{h}\|\abs{z_1-z_2} \\&\leq& k_0^{-1} h^{-1/2}\abs{\log h}\abs{z_1-z_2}\leq
C_5(n,\lambda,\Lambda, \Omega, u|_{\p\Omega}) r^{-1}\abs{\log r}\abs{z_1-z_2},\end{eqnarray*}
we find
\begin{equation}
\label{vbh2}
|v(z_1)-v( z_2)|  \le  |z_1-z_2|^{\beta} \quad \text{for all } z_1, z_2 \in  S_u(y, h/2),
\end{equation}
provided that $r=r_y\leq c_3<1$ is small.

{\it Step 2: Global H\"older estimates.}
We now combine (\ref{vbh2}) with (\ref{Bdr_assumed}) and (\ref{ballsecth2}) to prove $$\|v\|_{C^\beta(\bar \Omega)} \leq C(n,\lambda,\Lambda,\Omega, u|_{\p\Omega},\alpha,\delta, \gamma, K, M).$$
Indeed, as in (\ref{h-from-above}), there exists a constant $C_\ast(n,\lambda, M,\text{diam}(\Omega))$ such that
\begin{equation}
\label{vmax}
\|v\|_{L^{\infty}(\Omega)}\leq C_\ast.\end{equation}
It remains to estimate $\abs{v(x)-v (y)}/\abs{x-y}^{\beta}$ for $x$ and $y$ in $\Omega$. Let $r_{x} = \text{dist}(x, \p\Omega)$ and $r_{y}= \text{dist} (y, \p\Omega).$ Assume, without loss of generality, that 
$r_{y}\leq r_{x}$. Take $x_{0}\in\p\Omega$ and $ y_{0}\in \p\Omega$ such that \[r_{x}= \abs{x- x_{0}} \quad\text{and }r_{y} = \abs{y-y_{0}}.\]

From (\ref{vmax}) and the interior H\"older estimates in Theorem \ref{inter-H}, we only need to consider the case $r_{y}\leq r_{x}\leq c_3<1$. Consider the following cases.\\
{\it Case 1:
$\abs{x-y}\leq c_2 \frac{r_{x}}{\abs{\log r_{x}}}.$}
In this case, by (\ref{ballsecth2}), we have $$y\in B_{c_2 \frac{r_{x}}{\abs{\log r_{x}}}}(x)\subset S_u(x,h_x/2),$$
where
\[h_x=\sup\{t\,|~ S_u(x,t)\subset \Omega\}.\]
 In view of (\ref{vbh2}), we have
$$\frac{\abs{v(x)-v (y)}}{\abs{x-y}^{\beta}}\leq 1.$$
{\it Case 2:
$
\abs{x-y}\geq c_2 \frac{r_{x}}{\abs{\log r_{x}}}.$}
In this case, we have
\begin{equation}
\label{rxlogx}
r_{x}\leq c_2^{-1}\abs{x-y}\abs{\log \abs{x-y}}.
\end{equation}
Indeed, if 
$$1>r_{x}\geq \abs{x-y}\abs{\log \abs{x-y}}\geq \abs{x-y} $$
then
$$r_{x}\leq \frac{1}{c_2}\abs{x-y}\abs{\log r_{x}}\leq \frac{1}{c_2}\abs{x-y}\abs{\log \abs{x-y}}.$$
Due to (\ref{rxlogx}), we have
$$\abs{x_0-y_0}\leq r_{x} + \abs{x-y} + r_{y}\leq C_6(n,\lambda,\Lambda,\Omega, u|_{\p\Omega}) \abs{x-y}\abs{\log \abs{x-y}}.$$
Therefore, by (\ref{Bdr_assumed}),  $\|\varphi\|_{C^{\alpha}(\p\Omega)} \leq 1$, and $2\beta\leq \gamma\leq \alpha$, we obtain
\begin{eqnarray*}\abs{v(x)-v (y)}&\leq& \abs{v(x)- v(x_{0})} + 
\abs{v(x_{0})- v(y_{0})} + \abs{v(y_{0})- v(y)}\\ & \leq& C \left(r_{x}^{\gamma} + \abs{x_{0}- y_{0}}^{\alpha} + r_{y}^{\gamma}\right) \\ &\leq& C\left(\abs{x-y}\abs{\log \abs{x-y}}\right)^{\gamma}\leq C \abs{x-y}^{\beta},
\end{eqnarray*}
where $C=C(n,\lambda,\Lambda,\alpha, \Omega, u|_{\p\Omega},\delta, \gamma, K, M)$. This gives an estimate for $\abs{v(x)-v (y)}/\abs{x-y}^{\beta}$ in {\it Case 2}.

The proof of the theorem is complete.
\end{proof}

{\bf Acknowledgements. } The authors would like to thank the referee for carefully reading the paper and providing constructive comments that help improve the original manuscript.

\end{document}